\documentclass[12pt]{amsart}
\usepackage{amsmath,amsthm,amssymb}
\usepackage{enumerate}
\usepackage{graphicx}
\usepackage{float}
\usepackage{epstopdf}
\usepackage{cite}
\usepackage[margin=1in]{geometry}
\usepackage{tikz}
\usetikzlibrary{matrix}
\usetikzlibrary{arrows}
\usetikzlibrary{positioning}
\usepackage{overpic}
\usepackage{mathtools}
\usepackage{subcaption}

\theoremstyle{plain}
\newtheorem{thm}{Theorem}[section]
\newtheorem{cor}[thm]{Corollary}
\newtheorem{lem}[thm]{Lemma}

\newtheorem{conj}[thm]{Conjecture}
\theoremstyle{definition}

\theoremstyle{remark}
\newtheorem{rem}[thm]{Remark}
\newtheorem{ex}[thm]{Example}

\newcommand{\R}{\mathbb{R}}

\newcommand{\Z}{\mathbb{Z}}
\newcommand{\Q}{\mathbb{Q}}
\begin{document}

\title{Knot Floer Homology and Khovanov-Rozansky Homology for Singular Links}
\author{Nathan Dowlin}

\maketitle

\begin{abstract}

The (untwisted) oriented cube of resolutions for knot Floer homology assigns a complex $C_{F}(S)$ to a singular resolution $S$ of a knot $K$. Manolescu conjectured that when $S$ is in braid position, the homology $H_{*}(C_{F}(S))$ is isomorphic to the HOMFLY-PT homology of $S$. Together with a naturality condition on the induced edge maps, this conjecture would prove the spectral sequence from HOMFLY-PT homology to knot Floer homology. Using a basepoint filtration on $C_{F}(S)$, a recursion formula for HOMFLY-PT homology, and additional $sl_{n}$-like differentials on $C_{F}(S)$, we prove this conjecture. 

\end{abstract}

\section{Introduction}

The last few decades have seen tremendous growth within the field of knot theory. Many new knot invariants have been constructed, including the categorifications of several classical knot polynomials. These categorifications typically take the form of a multi-graded homology theory, whose graded Euler characteristic returns the polynomial in question. 

Some of the most notable categorifications include HOMFLY-PT homology, $sl_{n}$ homology, and knot Floer homology, whose graded Euler characteristics return the HOMFLY-PT polynomial, the $sl_{n}$ polynomial, and the Alexander polynomial, respectively. The HOMFLY-PT polynomial  of a link $L$ is a two variable polynomial $P_{H}(a,q)(L)$, and it is determined by the following skein relation:

\[ aP_{H}(a,q)(L_{+}) -a^{-1}P_{H}(a,q)(L_{-})= (q-q^{-1})P_{H}(a,q)(L_{0}) \]

\noindent
where $L_{+}$, $L_{-}$, and $L_{0}$ are identical except at one crossing, where $L_{+}$ has a positive crossing, $L_{-}$ has a negative crossing, and $L_{0}$ has the oriented smoothing (see \cite{HOMFLY} and \cite{PT}). Together with the normalization $P_{H}(a,q)(unknot) = 1$, this relation uniquely determines the HOMFLY-PT polynomial.

We can obtain a single-variable polynomial invariant $P_{n}(q)(L)$ by setting $a=q^{n}$ in the HOMFLY-PT polynomial. For $n \ge 1$, $ P_{n}(q)(L)$ gives the $sl_{n}$ polynomial, and setting $n=0$ gives the Alexander polynomial. The most popular of the $sl_{n}$ polynomials is the Jones polynomial, which is obtained by setting $n=2$, and $sl_{2}$ homology is isomorphic Khovanov's original categorification of the Jones polynomial, known as Khovanov homology (\hspace{1sp}\cite{Khov1}). An explicit description of this isomorphism is described in \cite{Hughes}.

The fact that the $sl_{n}$ and Alexander polynomials are specializations of the \linebreak HOMFLY-PT polynomial led to the following conjecture of Dunfield, Gukov, and Rasmussen:

\begin{conj} [\hspace{1sp}\cite{Gukov}]    \label{Gukthm}

For all $n \ge 1$, there is a spectral sequence from HOMFLY-PT homology to $sl_{n}$ homology, and there is a spectral sequence from HOMFLY-PT homology to knot Floer homology.

\end{conj}

Rasmussen was able to use the similar constructions of HOMFLY-PT homology and $sl_{n}$ homology to prove the first part of this conjecture. In particular, he showed that there are a family of spectral sequences $E_{k}(n)$ for $n \ge 1$ such that the $E_{2}$ page is HOMFLY-PT homology and the $E_{\infty}$ page is $sl_{n}$ homology \cite{Rasmussen}.

Unfortunately, due to the fundamental differences between Khovanov-Rozansky homology and knot Floer homology, the second part of Conjecture \ref{Gukthm} has remained unsolved for the last decade. This conjectured spectral sequence from HOMFLY-PT homology to knot Floer homology will be the focus of this paper.

It turns out that all three of these homology theories can be constructed via an oriented cube of resolutions. Let $C_{H}(D, d_{0}^{H} + d_{1}^{H}+...+d_{k}^{H})$ denote the cube of resolutions for HOMFLY-PT homology, where $d_{i}^{H}$ denotes the component of the differential which increases the cube grading by $i$. The HOMFLY-PT homology $H_{H}(K)$ is defined to be 

\[ H_{H}(K) = H_{*}(H_{*}(C_{H}(D), d_{0}^{H}), (d_{1}^{H})^{*})   \]

\noindent
where $D$ is a braid diagram of the knot $K$ and $(d_{1}^{H})^{*}$ is the induced map on homology. HOMFLY-PT homology is a knot invariant, which means that the homology does not depend on the choice of braid diagram for $K$. Note that this homology can be equivalently viewed as the $E_{2}$ page of the spectral sequence induced by the cube filtration on the HOMFLY-PT complex. 

The $sl_{n}$ homology $H_{n}(K)$ is defined in the same way. If $C_{n}(D, d_{0}^{(n)} + d_{1}^{(n)}+...+d_{k}^{(n)})$ denotes the cube of resolutions for $sl_{n}$ homology, then $H_{n}(K)$ is given by 

\[ H_{n}(K) = H_{*}(H_{*}(C_{n}(D), d_{0}^{(n)}), (d_{1}^{(n)})^{*})   \]

\noindent
where again, $D$ is a braid diagram for $K$, but the homology is independent of the choice of $D$.

Knot Floer homology is a completely different story. There is not a standard way to define the complex $\mathit{CFK}^{-}(K)$, as it depends on a choice of Heegaard diagram for $K$, and there are many different ways to do so for a knot $K$. However, the chain homotopy type of $\mathit{CFK}^{-}(K)$ does not depend on the choice of diagram. 

Using a particular choice of Heegaard diagram together with an algebraic component, Oszvath and Szabo developed an oriented cube of resolutions for knot Floer homology with twisted coefficients (\hspace{1sp}\cite{Szabo}). This construction was modified by Manolescu to give an untwisted cube of resolutions for knot Floer homology, which is chain homotopy equivalent to $\mathit{CFK}^{-}(K)$. We will denote this complex $(C_{F}(D), d_{0}^{F} + d_{1}^{F}+...+d_{k}^{F})$. Unlike the HOMFLY-PT and $sl_{n}$ homology, the knot Floer homology $\mathit{HFK}^{-}(K)$ is the \emph{total} homology of this complex.

\[   \mathit{HFK}^{-}(K) \cong H_{*}(C_{F}(D), d_{0}^{F} + d_{1}^{F}+...+d_{k}^{F})  \]

For all of these complexes, each vertex in the cube of resolutions can be viewed as a complex corresponding to the complete resolution $S$ at that vertex. We will denote these complexes by $C_{H}(S)$, $C_{n}(S)$, and $C_{F}(S)$, with the corresponding homologies given by $H_{H}(S)$, $H_{n}(S)$, and $H_{F}(S)$. 

Manolescu made the following conjecture:

\begin{conj} [\hspace{1sp}\cite{Manolescu}] \label{Man}

Let $S$ denote a complete resolution of a diagram $D$ in braid position. Then $ H_{H}(S) \cong H_{F}(S)$ as bigraded vector spaces.

\end{conj}

An immediate consequence of this conjecture is an isomorphism 

\[H_{*}(C_{H}(D), d_{0}^{H}) \cong H_{*}(C_{F}(D), d_{0}^{F})\]

\noindent
which maps each vertex in the HOMFLY-PT cube of resolutions to the same vertex in the knot Floer cube of resolutions. As discussed in \cite{Manolescu}, the spectral sequence from HOMFLY-PT homology to knot Floer homology would then follow from the induced edge maps being the same. In other words, if $f$ is the isomorphism between them, then the square below being commutative would prove the spectral sequence.

\begin{figure}[!h]
\centering
\begin{tikzpicture}
  \matrix (m) [matrix of math nodes,row sep=5em,column sep=6em,minimum width=2em] {
     H_{*}(C_{H}(D), d_{0}^{H}) & H_{*}(C_{F}(D), d_{0}^{F}) \\
     H_{*}(C_{H}(D), d_{0}^{H})& H_{*}(C_{F}(D), d_{0}^{F}) \\};
  \path[-stealth]
    (m-1-1) edge node [left] {$(d_{1}^{H})^{*}$} (m-2-1)
    (m-1-1) edge node [above] {$f$} (m-1-2)
    (m-1-2) edge node [right] {$(d_{1}^{F})^{*}$} (m-2-2)
    (m-2-1) edge node [above] {$f$} (m-2-2);
\end{tikzpicture}
\end{figure}

This idea can also be explained in terms of the spectral sequences induced by the cube filtrations on $C_{H}(D)$ and $C_{F}(D)$. Letting $E_{k}^{H}(D)$ and $E_{k}^{F}(D)$ denote the two spectral sequences, we see that the HOMFLY-PT homology is given by $E_{2}^{H}(D)$, and the knot Floer homology is given by $E_{\infty}^{F}(D)$. Manolescu's conjecture would imply an isomorphism $E_{1}^{H}(D) \cong E_{1}^{F}(D)$, and the induced edge maps commuting with this isomorphism would imply that $E_{2}^{H}(D) \cong E_{2}^{F}(D)$. This gives a spectral sequence whose $E_{2}$ page is isomorphic to HOMFLY-PT homology and whose $E_{\infty}$ page is $\mathit{HFK}^{-}(K)$.

Manolescu showed that for a singular braid $S$, both $H_{H}(S)$ and $H_{F}(S)$ have a purely algebraic formulation in terms of Tor groups. Letting $R$ denote the polynomial ring $\Q[U_{1},...,U_{k}]$, where $k$ is the number of edges in the singular braid $S$, he showed that there are ideals $L$, $Q$, and $N$ in $R$ such that 

\[H_{H}(S) \cong \mathit{Tor}_{R}(R/L, R/Q) \text{ \hspace{3mm}and\hspace{3mm} } H_{F}(S) \cong \mathit{Tor}_{R}(R/L, R/N) \]

\noindent

These Tor groups can naturally be viewed as bigraded vector spaces, where the dimension in each bigrading is finite. Thus, Conjecture \ref{Man} is equivalent to an isomorphism bigraded vector spaces

\[ \mathit{Tor}_{R}(R/L, R/Q) \cong \mathit{Tor}_{R}(R/L, R/N)  \]

\noindent
Unfortunately, these Tor groups turned out to be difficult to compare due to the non-local nature of the ideal $N$.

In this paper, we will prove Conjecture \ref{Man} using a very different approach. First, we define an additional family of differentials on $C_{F}(S)$ for all $n \ge 1$. We will denote this complex by $C_{F(n)}$. 

\begin{thm}

For all $n \ge 1$, there is an isomorphism 

\[H_{*}(C_{F(n)}(S)) \cong H_{n+1}(S) \]

\end{thm}

The differential on $C_{F(n)}$ can be filtered by the Alexander grading, and when we only consider those differentials which preserve the Alexander grading, we get back the complex $C_{F}(S)$. Thus, using the Alexander filtration on $C_{F(n)}$, we get the following corollary:

\begin{cor}

For all $n \ge 2$, there is a spectral sequence which starts at $H_{F}(S)$ and converges to $H_{n}(S)$.

\end{cor}

\noindent
For all $n \ge 1$, there is also a known spectral sequence which starts at $H_{H}(S)$ and converges to $H_{n}(S)$ (\hspace{1sp}\cite{Rasmussen}). Thus, $H_{H}(S)$ and $H_{F}(S)$ are both `limits' of $sl_{n}$ homology (in a suitable sense). 

We are able to use these additional differentials together with a basepoint filtration to prove Conjecture \ref{Man}.

\begin{thm}

Let $S$ denote a complete resolution of a diagram $D$ in braid position. Then $ H_{H}(S) \cong H_{F}(S)$ as bigraded vector spaces.

\end{thm}

\begin{cor}

Let $D$ be a braid diagram and $E^{F}_{2}(D)$ the $E_{2}$ page of the spectral sequence on $C_{F}(D)$ induced by the cube filtration. Then the graded Euler characteristic of $E^{F}_{2}(D)$ is the HOMFLY-PT polynomial.

\end{cor}

This corollary provides evidence for the conjecture that $E^{F}_{2}(D)$ is in fact isomorphic to HOMFLY-PT homology.

The isomorphisms $ H_{H}(S) \cong H_{F}(S)$ and $H_{*}(C_{F(n)}(S)) \cong H_{n+1}(S) $ are not canonical - they rely on a categorification of a relationship among knot polynomials known as the composition product. Some work in this area has been done by Wagner (\hspace{1sp}\cite{Wagner}), but we need the results in greater generality to prove the above theorems as bigraded vector spaces. These results are discussed in Section \ref{CompChapter}.

\section{The Khovanov-Rozansky Homology of Singular Links} \label{KRChapter}

\subsection{Singular Resolutions and the Ground Ring}

A complete resolution $S$ of a knot $K$ in braid position can be viewed as an oriented planar graph with the following properties:

1. All vertices are either 2-valent or 4-valent.

2. The number of incoming edges is equal to the number of outgoing edges at each vertex.

3. If $Z$ is an oriented cycle in $S$, then the unique disc $D \subset \R^{2}$ with boundary $Z$ intersects the center of the braid.

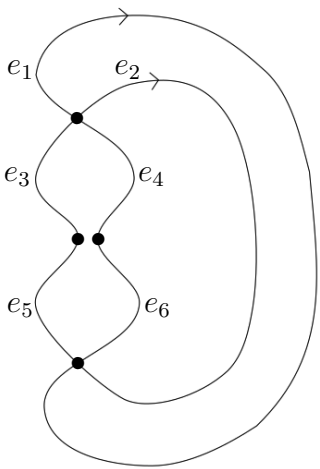
\begin{figure}[h!]
 \centering
   \begin{overpic}[width=.4\textwidth]{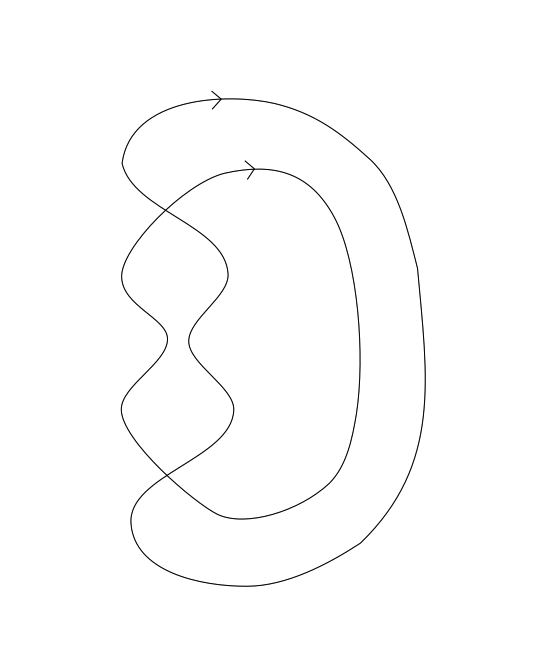}
   \put(24.4,66.2){$\bullet$}
   \put(14,75){$e_{1}$}
   \put(32,75){$e_{2}$}
   \put(13.5,57){$e_{3}$}
   \put(36,57){$e_{4}$}
   \put(14,35){$e_{5}$}
   \put(37,35){$e_{6}$}
   \put(24.6,25.1){$\bullet$}
   \put(24.6,46){$\bullet$}   
   \put(28,46){$\bullet$}   
   \end{overpic}
\caption{An example of a singular braid diagram}
\end{figure}

Let $e_{1},...,e_{k}$ denote the edges of $S$. To each edge $e_{i}$, we assign an indeterminant $U_{i}$. All three homology theories will be defined over the ground ring $R=\Q[U_{1},...,U_{k}]$.

\subsection{HOMFLY-PT Homology and $sl_{n}$ Homology} \label{HOMFLYsection}

This section will give a brief description of the HOMFLY-PT and $sl_{n}$ as defined in \cite{KR} and \cite{KR2}. We will use the grading conventions from \cite{Rasmussen}, though we will leave out the overall grading shifts coming from the braid number. The reader can refer to these resources for further background. The HOMFLY-PT and $sl_{n}$ complexes have the same generators, with the $sl_{n}$ complex having strictly more differentials than the HOMFLY-PT complex. For this reason, we will start by defining the HOMFLY-PT complex, then we will describe the additional differentials to make the $sl_{n}$ complex.

The HOMFLY-PT complex for links comes equipped with a triple-grading, and the $sl_{n}$ complex with a bigrading. One of the gradings in both theories, however, comes from the height in the cube of resolutions, so it will be fixed for a single resolution. The HOMFLY-PT complex will therefore come with a bigrading, and the $sl_{n}$ complex with a single grading. For the HOMFLY-PT complex, these gradings are called the quantum grading, denoted $gr_{q}$, and the horizontal grading, denoted $gr_{h}$. Multiplication by the $U_{i}$ increases the quantum grading by $2$ and preserves the horizontal grading.

Let $V_{2}(S)$ denote the 2-valent vertices in $S$ and $V_{4}(S)$ the 4-valent vertices of $S$. For vertices $v$ in $V_{2}(S)$, there is a unique outgoing edge $e_{i}$ and a unique incoming edge $e_{j}$. Define $L(v)$ to be the linear term $U_{i}-U_{j}$. Similarly, for vertices $v$ in $V_{4}(S)$ there are two outgoing edges $e_{i}$ and $e_{j}$ and two incoming edges $e_{k}$ and $e_{l}$. We define $L(v)$ to be the linear term $U_{i}+U_{j}-U_{k}-U_{l}$ and $Q(v)$ to be the quadratic term $U_{i}U_{j}-U_{k}U_{l}$.

The HOMFLY-PT complex is a tensor product of complexes $C_{H}(v)$ for each vertex $v$. For $v$ in $V_{2}(S)$, $C_{H}(v)$ is defined as 

\begin{figure}[!h]
\centering
\begin{tikzpicture}
  \matrix (m) [matrix of math nodes,row sep=5em,column sep=6em,minimum width=2em] {
     R\{0, -2\} & R\{0, 0\} \\};
  \path[-stealth]
    (m-1-1) edge node [above] {$L(v)$} (m-1-2);
\end{tikzpicture}
\end{figure}
 
\noindent
where $R\{i,j\}$ refers to the thing $R$ shifted by $i$ in $gr_{q}$ and by $j$ in $gr_{h}$. For $v$ in $V_{4}(S)$, $C_{H}(v)$ is defined as 

\begin{figure}[!h]
\centering
\begin{tikzpicture}
  \matrix (m) [matrix of math nodes,row sep=5em,column sep=6em,minimum width=2em] {
     R\{1,-4\} & R\{1,-2\} \\
     R\{-1,-2\} & R\{-1,0\} \\};
  \path[-stealth]
    (m-1-1) edge node [left] {$Q(v)$} (m-2-1)
    (m-1-1) edge node [above] {$L(v)$} (m-1-2)
    (m-1-2) edge node [right] {$Q(v)$} (m-2-2)
    (m-2-1) edge node [above] {$L(v)$} (m-2-2);
\end{tikzpicture}
\end{figure}

\noindent
Note that the differential is homogeneous of degree $\{2, 2\}$. The HOMFLY-PT complex for the singular diagram $S$ is given by

\[ C_{H}(S) = \bigotimes_{v \in S} C_{H}(v)    \]

\noindent
where the tensor product is taken over $R$, and the HOMFLY-PT homology $H_{H}(S)$ is the homology of $C_{H}(S)$.

We will now define the additional differentials which give $sl_{n}$ homology. For a vertex $v$ in $S$ with outgoing edges $E_{out}$ and incoming edges $E_{in}$, let the potential $w_{n}$ be given by 

\[ w_{n}(v) = \sum_{e_{i} \in E_{out}} U_{i}^{n+1} -  \sum_{e_{j} \in E_{in}} U_{j}^{n+1}   \]

\noindent
For $v$ in $V_{2}(S)$, let $u_{1}(v)$ be the unique element in $R$ such that $u_{1}(v)L(v)=w_{n}(v)$. For $v$ in $V_{4}(S)$, we can choose $u_{1}(v)$ and $u_{2}(v)$ such that $u_{1}(v)L(v)+u_{2}(v)Q(v)=w_{n}(v)$. Unlike the 2-valent case, the choice is not unique, but the reader can refer to (\hspace{1sp}\cite{KR}, p. 6) for the precise choice. (It is not relevant for our discussion.)

For each vertex $v$, we will add new differentials to $C_{H}(v)$ to make a new complex $C_{n}(V)$. For $v$ in $V_{2}(S)$, $C_{n}(v)$ is given by 

\begin{figure}[!h]
\centering
\begin{tikzpicture}
  \matrix (m) [matrix of math nodes,row sep=5em,column sep=8em,minimum width=2em] {
     R\{0, -2\} & R\{0, 0\} \\};
  \path[-stealth]
    (m-1-1) edge [bend left=15] node [above] {$L(v)$} (m-1-2)
    (m-1-2) edge [bend left=15] node [below] {$u_{1}(v)$} (m-1-1);
\end{tikzpicture}
\end{figure}

\noindent
and for $v$ in $V_{4}(S)$, $C_{n}(v)$ is given by 

\begin{figure}[!h]

\centering
\begin{tikzpicture}
  \matrix (m) [matrix of math nodes,row sep=7em,column sep=8em,minimum width=2em] {
     R\{1,-4\} & R\{1,-2\} \\
     R\{-1,-2\} & R\{-1,0\} \\};
  \path[-stealth]
    (m-1-1) edge [bend right = 15] node [left] {$Q(v)$} (m-2-1)
            edge [bend left = 15] node [above] {$L(v)$} (m-1-2)
    (m-2-1) edge [bend left = 15] node [above] {$L(v)$} (m-2-2)
            edge [bend right = 15] node [right] {$u_{2}(v)$} (m-1-1)    
    (m-1-2) edge [bend right=15] node [left] {$Q(v)$} (m-2-2)
            edge [bend left=15] node [below] {$u_{1}(v)$} (m-1-1)
    (m-2-2) edge [bend left=15] node [below]  {$u_{1}(v)$} (m-2-1)
            edge [bend right=15] node [right] {$u_{2}(v)$} (m-1-2);

\end{tikzpicture}
\end{figure}

Observe that for both types of vertices, the differential on $C_{n}(v)$ satisfies $d^{2}=w_{n}(v)I$. Such a complex is called a matrix factorization with potential $w_{n}$. Since $d^{2}$ is non-zero, its homology is not well-defined. However, we are interested in the tensor product of $C_{n}(v)$ over all vertices $v$ in $S$. Define the $sl_{n}$ complex $C_{n}(S)$ by

\[ C_{n}(S) = \bigotimes_{v \in S} C_{n}(v)    \]

\noindent
where again the tensor product is taken over $R$.

As mentioned above, the HOMFLY-PT differentials are homogeneous of degree $\{2, 2\}$. These differentials are denoted by $d_{+}$. The new differentials, those with coefficients $u_{1}(v)$ and $u_{2}(v)$, are homogeneous of degree $\{2n, -2\}$. These are denoted $d_{-}$. The total differential $d_{tot} = d_{+}+d_{-}$ is not homogeneous in this bigrading. However, if we look at the grading $gr_{n} = gr_{q}+(n-1)gr_{h}/2$, then $d_{tot}$ is homogeneous of degree $n+1$. 

Additionally, $d_{tot}^{2}=0$. This can be seen from the fact that the potential is additive under tensor product, so $d_{tot}^{2}= \sum_{v \in S} w_{n}(v)$. The sum must be zero because each edge $e_{i}$ is an outgoing edge for one vertex, which will contribute $U_{i}^{n+1}$, and an incoming edge for another vertex, which will contribute $-U_{i}^{n+1}$.

We have shown that $C_{n}(S)$ is a well-defined chain complex which is homogeneous with respect to the grading $gr_{n}$. We define the $sl_{n}$ homology $H_{n}(S)$ to be the homology of this complex. 

\begin{rem}

The definitions given here correspond to the unreduced theories in \cite{Rasmussen} as opposed to the middle or reduced homologies. This choice will make our arguments involving the composition product easier, although they would still work for the reduced versions using the destabilized composition product, which is described in Section \ref{dest}. 

\end{rem}

\subsection{Rasmussen's Spectral Sequences} \label{sect2.3}

Rasmussen showed in \cite{Rasmussen} that there are a family of spectral sequences $E_{k}(n)$ which start at HOMFLY-PT homology and converge to $sl_{n}$ homology. These spectral sequences are somewhat difficult to prove for the case of knots and links, but they are much simpler for fully singular diagrams.

The differential $d_{tot}$ on $C_{n}(S)$ admits a filtration induced by the grading $(gr_{q}-gr_{h})/(2n+2)$. With respect to this grading, $d_{+}$ is homogeneous of degree 0, and $d_{-}$ is homogeneous of degree 1. Note that this filtration is not bounded above, so \emph{a priori} it does not induce a well-defined spectral sequence to the total homology. However, if we look at the induced differentials which change this grading by $k$, we see that they also decrease the horizontal grading by $2-4k$. Since the complex is bounded in horizontal grading, this bounds the length of induced differentials, so the filtration does induce a spectral sequence to the total homology.

The $E_{1}$ page of this spectral sequence is $H(C_{n}(S), d_{+})$, which is exactly the definition of HOMFLY-PT homology. The $E_{\infty}$ page is the homology with respect to $d_{tot}$, or $sl_{n}$ homology.

It turns out that this spectral sequence collapses at the $E_{2}$ page, given by

\[H_{*}(H_{*}(C_{n}(S), d_{+}),d_{-}^{*})\]

\noindent
We will denote this page by $H^{\pm}(C_{n}(S))$. Note that $H^{\pm}(C_{n}(S))$ is bigraded, as both $d_{+}$ and $d_{-}$ are homogeneous. The fact that all higher differentials are trivial follows from the following lemma.

\begin{lem}[\hspace{1sp}\cite{Rasmussen}]

The homology $H^{\pm}(C_{n}(S))$ lies in a single horizontal grading, namely $gr_{h}=2r(S)$, where $r(S)$ is the rotation number of $S$. Since $S$ is a singular braid oriented clockwise, $r(S)$ is the negative of the number of strands in $S$.

\end{lem}

Since none of the higher differentials preserve the horizontal grading, they must all be trivial, causing the spectral sequence to collapse. 

\begin{cor}\label{gradinglemma}

Viewing $H^{\pm}(C_{n}(S))$ as singly graded with grading $gr_{n}$, there is a graded isomorphism $H^{\pm}(C_{n}(S)) \cong H_{n}(S)$.

\end{cor}

\begin{rem}

The reader familiar with \cite{Rasmussen} may note that our homology lies in $gr_{h}=-2r(S)$, while Rasmussen's lies in $gr_{h}=1-r(S)$. This difference comes from the fact that our homology is unreduced, which changes the grading by $2$, and because we are leaving out the overall grading shift of $r(S)-1$.

\end{rem}

\section{A Recursion Formula for the Khovanov-Rozansky Homology of Singular Links} \label{CompChapter}

\subsection{The Composition Product} \label{clearcycle}

\subsubsection{Jaeger's Formula}

The first composition product formula was defined by Jaeger in \cite{Jaeger}. In order to talk about the composition product, we must first define labelings of a diagram. Let $K$ be a knot with corresponding diagram $D$. Viewing $D$ as an oriented 4-valent graph, we say that a subset $S$ of the edges of $D$ is a homological cycle if at each vertex in $D$ the number of incoming edges in $S$ is equal the the number of outgoing edges in $S$. A \emph{labeling} $f$ of the diagram $D$ is a function from the set of edges in $D$ to the set $\{1,2\}$ such that $f^{-1}(1)$ is a homological cycle. (Note that  $f^{-1}(1)$ is a homological cycle iff $f^{-1}(2)$ is a homological cycle.)

We will place a restriction on which homological cycles are allowed. A homological cycle is said to make a turn at a crossing $c$ if it has one incoming edge at $c$ and one outgoing edge at $c$, and those edges are not diagonal from one another. Let $T(f)$ denote the number of turns of the labeling $f$. A labeling $f$ is \emph{admissible} if $f^{-1}(1)$ doesn't make any left turns at positive crossings or right turns at negative crossings.

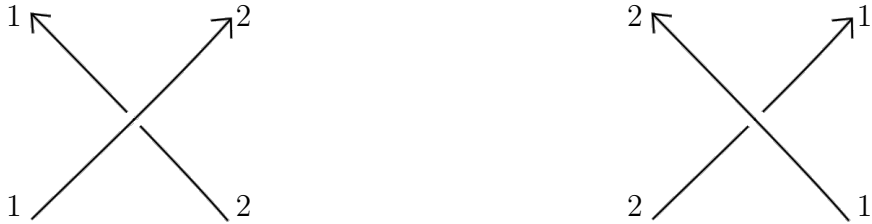
\begin{figure}[h!]
\centering
\begin{subfigure}{.5\textwidth}
  \centering
   \begin{overpic}[width=.5\textwidth]{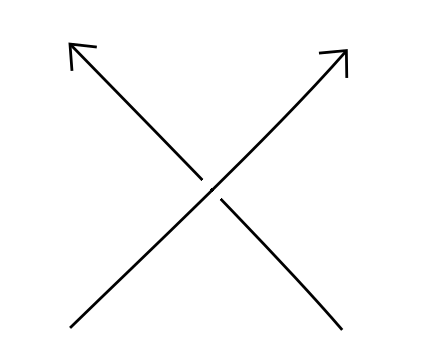} 
   \put(8,9){$1$}
   \put(8,70){$1$}
   \put(82,9){$2$}
   \put(82,70){$2$}
   \end{overpic}
\end{subfigure}%
\begin{subfigure}{.5\textwidth}
  \centering
  \begin{overpic}[width=.5\textwidth]{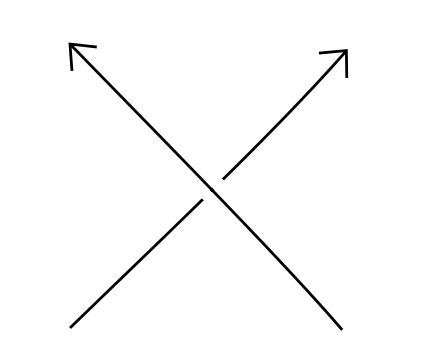} 
   \put(8,9){$2$}
   \put(8,70){$2$}
   \put(82,9){$1$}
   \put(82,70){$1$}
  \end{overpic}
\end{subfigure}
\caption{Non-Admissible Labelings}
\end{figure}

Since the homological cycle $f^{-1}(1)$ uniquely determines the labeling $f$, we will say that a homological cycle $Z$ is admissible if the unique labeling $f$ with $f^{-1}(1)=Z$ is admissible. The two cycles $f^{-1}(1)$ and $f^{-1}(2)$ can both be viewed as diagrams of links if we retain the crossing information whenever one of them contains all four edges at a crossing, and forget it otherwise. We will refer to these diagrams as $D_{f,1}$ and $D_{f,2}$, respectively. 

Finally, we will need some combinatorial data about the diagrams $D_{f,1}$ to define the composition product. Given a diagram $D$, consider the diagram obtained by changing each crossing in $D$ to the oriented smoothing. The resulting diagram must be a collection of oriented circles - these are known as the Seifert circles of $D$. We define the \emph{rotation number} $r(D)$ to be sum of the signs of the Seifert circles, with a circle contributing a $+1$ if it is oriented counterclockwise and $-1$ if it is oriented clockwise. Note that when $D$ is a braid, $r(D)$ is simply the negative of the number of strands in the braid, i.e. $r(D)=-b(D)$.

With the HOMFLY-PT polynomial $P_{H}$ as defined in the introduction, define

\[P'_{H}(a,q,D)=(\frac{a-a^{-1}}{q-q^{-1}})(a^{w(D)})P_{H}(a,q,D)\]

\noindent
Note that $P'_{H}$ is invariant under Reidemeister II and III moves, but performing a Reidemeister I move changes the writhe, so one picks up a factor of $a$ or $a^{-1}$ depending on the sign of the crossing. With this normalization, $P'_{H}(unknot)=\frac{a-a^{-1}}{q-q^{-1}}$, and $P'_{H}(\emptyset)=1$, where $\emptyset$ denotes the empty diagram. Jaeger's composition product can be stated as follows:

\begin{equation}
\label{Jaeger}
\begin{split}
 \hspace{-5mm}\mathop{\sum_{f \text{ admissible}}} (q-q^{-1})^{T(f)}a_{1}^{r(D_{f,2})}a_{2}^{-r(D_{f,1})} & P'_{H}(a_{1},q,D_{f,1})P'_{H}(a_{2},q,D_{f,2}) \\
 & = P'_{H}(a_{1}a_{2},q,D)
\end{split} 
\end{equation}

The proof of this formula is combinatorial in nature - one can show that it behaves properly under Reidemeister moves and that it satisfies the necessary skein relation via local computations. In fact, Jaeger showed in \cite{Jaeger} that this formula is invariant under all Reidemeister moves, so the formula holds for arbitrary diagrams $D$ instead of just braid diagrams.

To complete the proof, we just have to check that it works on the unknot, which is calculated below.

\[ a_{2}^{-1} \frac{a_{1}-a_{1}^{-1}}{q-q^{-1}} + a_{1} \frac{a_{2}-a_{2}^{-1}}{q-q^{-1}} = \frac{a_{1}a_{2}-a_{1}^{-1}a_{2}^{-1}}{q-q^{-1}}=P_{H}'(a_{1}a_{2}, q, unknot) \]

\subsubsection{The Destabilized Composition Product}
\label{dest}

We develop an adaptation of Jaeger's composition product for a decorated braid diagram $D$, i.e. a diagram in braid position that has one marked edge $e_{0}$, in the special case where $a_{1}=q$. (Again, the condition that our diagram is a braid is not necessary for this definition, but it will be necessary on the level of homology.) The diagram $D$ now has a special Seifert circle, the one containing the marked edge $e_{0}$ - we will call this circle $S_{0}$. We will define the sign of a Seifert circle $S$ as follows:

\begin{equation}
\label{signs}
 sign(S) = \begin{cases} 
      +1 & \textrm{ if $S$ is oriented CCW and $S$ does not contain $S_{0}$} \\
      -1 & \textrm{ if $S$ is oriented CCW and $S$ contains $S_{0}$} \\
      -1 & \textrm{ if $S$ is oriented CW and $S$ does not contain $S_{0}$} \\
      +1 & \textrm{ if $S$ is oriented CW and $S$ contains $S_{0}$} \\
      0 & \textrm{ if $S=S_{0}$} \\
      
   \end{cases} 
\end{equation}

An alternative way to view these signs is to imagine our diagram is in $S^{2}$ instead of the plane, so that each Seifert circle bounds two discs. To determine the sign of the Seifert circle, we view it as the boundary of the disc that does not contain the edge $e_{0}$. Then, as before, we say that it is $+1$ if it is oriented CCW and negative if it is oriented CW. The special circle containing $e_{0}$ has no contribution.

With these sign conventions, let $r(D)$ denote the sum of the signs of the Seifert circles. We can define our reduced version of the composition product by 

\begin{equation}
\label{compred}
\begin{split}
  \mathop{\sum_{f \text{ admissible}}}_{f(e_{0})=2} \Big[ (q-q^{-1})^{T(f)} & q^{r(D_{f,2})-s(D_{f,2})} a^{-r(D_{f,1})-s(D_{f,1})} \\
& \cdot  P_{H}(q,q,D_{f,1})P_{H}(a,q,D_{f,2}) \Big] = P_{H}(qa,q,D)
\end{split}  
\end{equation}

\noindent
where the signs of the Seifert circles in both $D_{f,1}$ and $D_{f,2}$ are given by (\ref{signs}) relative to the marked edge $e_{0}$, even though $e_{0}$ always belongs to $D_{f,2}$. The quantities $s(D_{f,i})$, defined to be the sum of the signs of the crossings in $D$ with at least one adjacent edge labeled $i$, stems from the fact that Jaeger's composition product came with factors of $a^{w(D)}$. (Note that $s(D_{f,1})=w(D)-w(D_{f,2})$ and  $s(D_{f,2})=w(D)-w(D_{f,1})$.)

The proof of this equality is identical to the proof for Jaeger's, since everything is the same locally. The only differences (aside from the notational difference of removing the shifts by $w(D)$) are that our labelings require that $f(e_{0})=2$ and we don't have the factor of $\frac{a-a^{-1}}{q-q^{-1}}$. Jaeger's calculations show that our construction satisfies the correct skein relation, and that it is invariant under Reidemeister moves that take place away from the marked edge $e_{0}$. By the equivalence of knots and (1,1) tangles, these are the only Reidemeister moves we need to show invariance. Thus, to complete the proof, we just need to check that the formula holds on the base case of the unknot. 

\begin{rem}

It is important to note that the equivalence of knots and (1,1) tangles is not true when restricted to braids - one must sometimes isotope through non-braid diagrams in order to connect two equivalent tangles. Fortunately, Jaeger showed in \cite{Jaeger} that Equation \ref{Jaeger} is invariant under all Reidemeister moves regardless of orientation, giving an invariant for arbitrary diagrams.

\end{rem}

There is only one labeling that contributes for the unknot. Since there is only one edge, it must be the marked edge $e_{0}$, and $f(e_{0})=2$. For this labeling $f$, $r(D_{f,1})=r(D_{f,2})=s(D_{f,1})=s(D_{f,2})=0$, $P_{1}(\phi)=\frac{q-q^{-1}}{q-q^{-1}}=1$, and $P_{H}(unknot)=1$, so (\ref{compred}) becomes $1 \cdot 1=1$. This establishes the base case, which proves the formula.

\begin{ex} \textbf{The Right Handed Trefoil}

\noindent

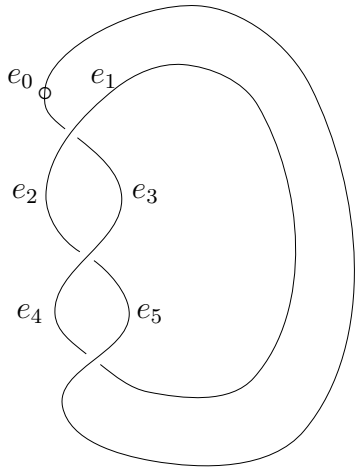
\begin{figure}[h!]

 \centering
   \begin{overpic}[width=.25\textwidth]{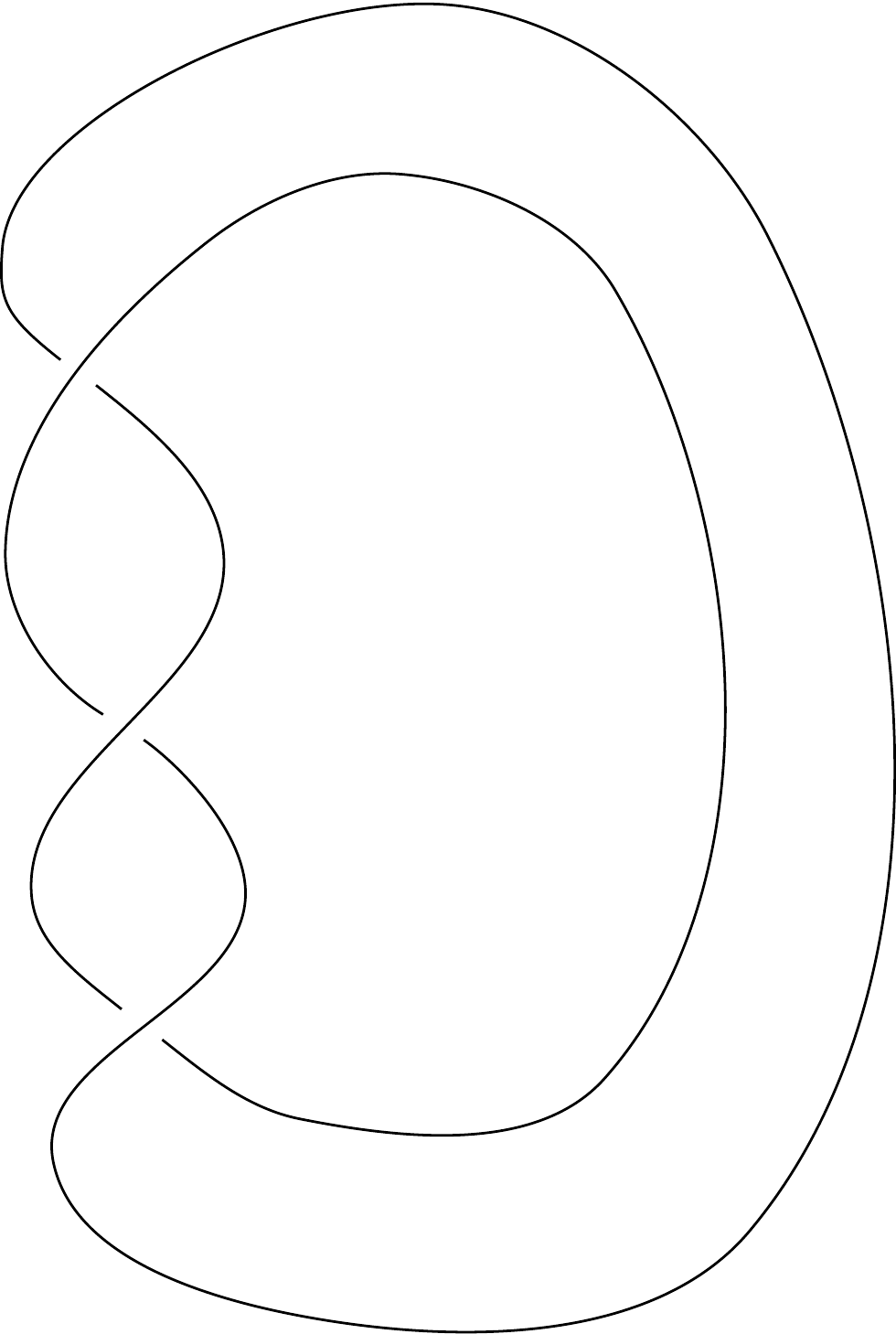}
   \put(-8,83){$e_{0}$}
   \put(10,83){$e_{1}$}
   \put(-7,58){$e_{2}$}
   \put(19,58){$e_{3}$}
   \put(-6,32){$e_{4}$}
   \put(20,32){$e_{5}$}
   \put(-1.5,79){$\circ$}
   \end{overpic}
   
\caption{Braid Diagram for the Right Handed Trefoil}

\end{figure}

This diagram $D$ has four local homological cycles which we will describe in terms of the edges in $f^{-1}(1)$, since this set uniquely characterizes $f$. These four sets are $\emptyset$, $e_{1}e_{2}e_{5}$, $e_{1}e_{3}e_{4}$, and $e_{1}e_{3}e_{5}$. Their contributions are listed in the table below. The sum of these contributions is $a^{-2}+a^{-2}q^{-4}-a^{-4}q^{-4}$, which is equal to $P_{H}(qa,q,D)$. We can see this equality by substituting $a \mapsto aq$ in the contribution from $\emptyset$.

\begin{center}
  \begin{tabular}{ l | c }

    Cycle & Contribution  \\ \hline
    $\emptyset$ & $a^{-2}q^{2}+a^{-2}q^{-2}-a^{-4}$ \\ \hline
    $e_{1}e_{2}e_{5}$ & $(q-q^{-1})q^{-3}a^{-2}$  \\ \hline
    $e_{1}e_{3}e_{4}$ & $(q-q^{-1})q^{-3}a^{-2}$  \\ \hline
    $e_{1}e_{3}e_{5}$ & $(q-q^{-1})^{3}q^{-3}a^{-2}$  \\ \hline
    \bf{Total} & $a^{-2}+a^{-2}q^{-4}-a^{-4}q^{-4}$ \\ \hline

  \end{tabular}
\end{center}

\end{ex}

\subsubsection{The Composition Product for Singular Graphs}

The composition product formula can be specialized to the $sl_{n}$ polynomials to give 

\begin{equation}
\label{slncomp}
 \mathop{\sum_{f \text{ admissible}}} (q-q^{-1})^{T(f)}q^{mr(D_{f,1})-nr(D_{f,2})} P'_{n}(q,D_{f,1})P'_{m}(q,D_{f,2}) = P'_{m+n}(q,D)
\end{equation}

This formula was extended by Wagner to singular braids in the following way. If $S$ is a singular braid, we can define labelings of $S$ in the same way as labelings for knots. We will drop the admissibility condition at 4-valent vertices since they no longer correspond to positive or negative crossings. Finally, given a labelling $f$ of $S$, let $T_{1}(S_{f,1})$ denote the number of vertices $v \in V_{4}(S)$ at which $f^{-1}(1)$ contains the edges $e_{1}$ and $e_{3}$ in Figure \ref{vertex}.

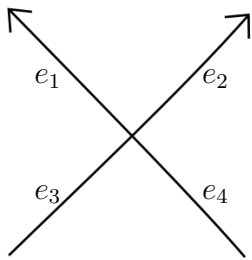
\begin{figure}[h!]
 \centering
   \begin{overpic}[width=.3\textwidth]{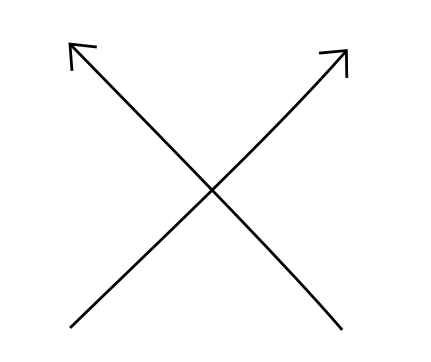}
   \put(23,54){$e_{1}$}
   \put(68,54){$e_{2}$}
   \put(23,23){$e_{3}$}
   \put(68,23){$e_{4}$}
   \end{overpic}
\caption{A labeled 4-valent vertex}
\label{vertex}
\end{figure}

\noindent
Similarly, let $T_{2}(S_{f,1})$ denote the number of vertices $v \in V_{4}(S)$ at which $f^{-1}(1)$ contains the edges $e_{2}$ and $e_{4}$. With this terminology, the composition product for singular braids can be stated as 

\[ P'_{m+n}(S) = \sum_{f \in L(S)} q^{\sigma_{m,n}(f)} P'_{n}(S_{f,1})P'_{m}(S_{f,2})    \]

\noindent
where $P'_{k}(S)$ is the unreduced $sl_{k}$ polynomial of $S$, $\sigma_{m,n}(f)=T_{2}(S_{f,1}) - T_{1}(S_{f,1}) + mr(S_{f,1}) - nr(S_{f,2})$, and $r(S_{f,i})$ is the negative of the number of strands in the singular braid $S_{f,i}$.

\subsection{The Categorification for $sl_{n}$ Homology} \label{bigradedcomp}

Wagner further showed that this relationship is true on the level of categorifications

\[ H_{m+n}(S) = \bigoplus_{f \in L(S)}  H_{n}(S_{f,1}) \otimes H_{m}(S_{f,2}) \{ \sigma_{m,n}(f) \}    \hspace{8mm} (\text{with polynomial grading} ) \]

\noindent
where the $sl_{n}$ homology groups are singly graded, using the polynomial grading $gr_{n}$. In order for us to generalize this theorem, it will be useful to have a bigraded version of it using our $(gr_{q}, gr_{h})$ gradings. Since the bigraded $H_{n}(S)$ lies in a single horizontal grading, namely $gr_{h} = 2r(S)$, and $gr_{n} = gr_{q} + (n-1)gr_{h}/2$, the $q$-grading is uniquely determined by $gr_{n}$:

\[ gr_{n} = gr_{q} + (n-1)r(S) \]

Thus, if we view each $H_{n}(S)$ as singly graded where the grading is $gr_{q}$ instead, the formula becomes 

\[ H_{m+n}(S) \{(n+m-1)r(S)\} =  \hspace{8mm} (\text{with quantum grading} ) \]
\[ \bigoplus_{f \in L(S)}  H_{n}(S_{f,1}) \otimes H_{m}(S_{f,2}) \{ \sigma_{m,n}(f)+(n-1)r(S_{f,1})+(m-1)r(S_{f,2}) \}    \]

\noindent
This can be simplified to 

\[ H_{m+n}(S) = \bigoplus_{f \in L(S)}  H_{n}(S_{f,1}) \otimes H_{m}(S_{f,2}) \{ T_{2}(S_{f,1})- T_{1}(S_{f,1}) - 2nr(S_{f,2})\} \]

Finally, to make this bigraded, we need to add in the horizontal grading. Since $H_{m+n}(S)$ lies in horizontal grading $ 2r(S)$, $H_{n}(S_{f,1})$ lies in horizontal grading $r(S_{f,1})$, and $H_{m}(S_{f,2})$ lies in horizontal grading $r(S_{f,2})$, we see that the tensor product

\[H_{n}(S_{f,1}) \otimes H_{m}(S_{f,2}) \]

\noindent
always lies in horizontal grading $2r(S_{f,1})+2r(S_{f,2})$. Since $r(S_{f,1})+r(S_{f,2})=r(S)$, it follows that the tensor product lies in grading $2r(S)$ - the same grading as $H_{m+n}(S)$. Thus, we can add the horizontal grading to the composition product formula, with no grading shift required for each labeling $f$. The bigraded formula is then 

\[ H_{m+n}(S) = \bigoplus_{f \in L(S)}  H_{n}(S_{f,1}) \otimes H_{m}(S_{f,2}) \{ T_{2}(S_{f,1})- T_{1}(S_{f,1}) - 2nr(S_{f,2}), 0\} \]

\noindent
where the bigrading is given by $(gr_{q}, gr_{h})$.

\subsection{A Categorification for HOMFLY-PT Homology}

In relating these formulas to knot Floer homology, we will be most interested in the case when $n=1$. In this case, the previous formula becomes 

\begin{equation} \label{wagnercomp}
 H_{m+1}(S) = \bigoplus_{f \in L(S)}  H_{1}(S_{f,1}) \otimes H_{m}(S_{f,2}) \{ T_{2}(S_{f,1})- T_{1}(S_{f,1}) - 2r(S_{f,2}), 0\} 
\end{equation}

In this section we will prove a generalization of this formula to HOMFLY-PT homology. Letting $H_{H}(S)$ denote the HOMFLY-PT homology with the standard bigrading $(gr_{q},gr_{h})$, define $H_{H}(S)<k>$ to be HOMFLY-PT homology with a new grading $(gr_{q}+kgr_{h}, gr_{h})$.

\begin{thm} \label{thm2.1}

There is an isomorphism of bigraded vector spaces 
\[ \bigoplus_{f} H_{1}(S_{f,1}) \otimes H_{H}(S_{f,2})\{  T_{2}(S_{f,1}) - T_{1}(S_{f,1})-2r(S_{f,2}),  0 \}  \cong H_{H}(S)<1> \]

\end{thm}

The proof of this theorem will rely heavily the fact that for all $n \ge 1$, there is a differential $d_{-}(n)$ on $H_{H}(S)$ which is homogeneous with bigrading $\{2n,-2\}$, and 

\[H_{*}(H_{H}(S), d_{-}(n)) \cong H_{n}(S) \]

\noindent
where we are viewing $H_{n}(S)$ as a bigraded vector space. (This is the homology $H^{\pm}_{n}(S)$ from Section \ref{sect2.3}.) Recall that as a bigraded vector space, $H_{n}(S)$ lies in a single horizontal grading, namely $gr_{h} = 2r(S)$.

The theorem will be proved in two parts. First we will show that any bigraded vector space with certain algebraic properties must be isomorphic to $H_{H}(S)<k>$, and then we will show that our construction satisfies those properties for $k=1$.

\begin{lem} \label{thelemma}

Let $H(S)$ denote a bigraded vector space with the following properties:
\vspace{2mm}

1) $H(S)$ is bounded above and below in horizontal grading.

2) $H(S)$ is bounded below in q-grading.

3) $H(S)$ is finite dimensional in each bigrading.

4) For all $n \ge n_{0}$, there is a differential $d_{n}$ on $H(S)$ which is homogeneous with bigrading $\{2n,-2\}$ such that $H_{*}(H(S), d_{n}) \cong H_{n}(S)$, where the isomorphism is as bigraded vector spaces.
\vspace{2mm}

Then $H(S) \cong H_{H}(S)$.

\end{lem}

Note that HOMFLY-PT homology itself satisfies these conditions, so they are not vacuous. 

\begin{proof}

Let $H^{i,j}$ denote the homology in bigrading $\{i,j\}$, and similarly for $H_{H}^{i,j}(S)$. The lemma states that for any integers $i$ and $j$, $dim(H^{i,j}(S)) = dim(H_{H}^{i,j}(S))$.

Suppose for some $i,j$, the dimensions do not agree. Since $H(S)$ is bounded in horizontal grading, we can take choose $i_{0}$ and $j_{0}$ so that $j_{0}$ is minimized subject to the constraint that $dim(H^{i_{0},j_{0}}(S)) \ne dim(H^{i_{0},j_{0}}_{H}(S))$. Note that the choice for $i_{0}$ may not be unique.

Since both $H(S)$ and $H_{H}(S)$ are bounded below in $q$-grading, there exists a constant $a(S)$ such that for all $i \le a(S)$, $dim(H^{i,j}(S)) = dim(H_{H}^{i,j}(S)) = 0$

Choose $n > max(i_{0}-a, n_{0})$. We will now use the fact that $H_{*}(H(S), d_{n}) \cong H_{n}(S)$ as bigraded vector spaces. We can rewrite $sl_{n}$ homology as the homology of $H_{H}(S)$ with respect to $d_{-}(n)$.

\[ H_{*}(H(S), d_{n}) \cong H_{*}(H_{H}(S), d_{-}(n)) \]

The differentials on both of these complexes have bigrading $\{2n,-2\}$. We can put an equivalence relation on $\Z^{2}$ where $(i,j) \sim (i', j')$ if $(i-i', j-j')=k(2n,-2)$, and both complexes must split according to this equivalence relation. In other words, for a fixed equivalence class $A$, the sum 

\[ \bigoplus_{(i,j) \in A} H^{i,j}(S) \]

\noindent
gives a subcomplex of $(H(S), d_{n})$, and similarly for $(H_{H}(S), d_{-}(n))$. 

Consider the summand corresponding to equivalence class of $(i_{0},j_{0})$. For the complex $(H(S), d_{n})$, this summand looks like

\[ ... \xrightarrow{\hspace{1.5mm}d_{n}\hspace{1.5mm}} H^{i_{0}-2n,j_{0}+2}(S)  \xrightarrow{\hspace{1.5mm}d_{n}\hspace{1.5mm}} H^{i_{0},j_{0}}(S) \xrightarrow{\hspace{1.5mm}d_{n}\hspace{1.5mm}} H^{i_{0}+2n,j_{0}-2}(S)      \xrightarrow{\hspace{1.5mm}d_{n}\hspace{1.5mm}}  ...               \]

\vspace{5mm}
\noindent
and for $(H_{H}(S), d_{-}(n))$,
\[ ... \xrightarrow{d_{-}(n)} H_{H}^{i_{0}-2n,j_{0}+2}(S)  \xrightarrow{d_{-}(n)} H_{H}^{i_{0},j_{0}}(S) \xrightarrow{d_{-}(n)} H_{H}^{i_{0}+2n,j_{0}-2}(S)      \xrightarrow{d_{-}(n)}  ...               \]

\vspace{2mm}
Now, since the complex is bounded below in $q$-grading and we've chosen $n$ sufficiently large, all of the chain groups before the $(i_{0},j_{0})$ summand are trivial, so the complexes become 

\[ ...0 \xrightarrow{\hspace{1.5mm}d_{n}\hspace{1.5mm}} 0  \xrightarrow{\hspace{1.5mm}d_{n}\hspace{1.5mm}} H^{i_{0},j_{0}}(S) \xrightarrow{\hspace{1.5mm}d_{n}\hspace{1.5mm}} H^{i_{0}+2n,j_{0}-2}(S)      \xrightarrow{\hspace{1.5mm}d_{n}\hspace{1.5mm}}  ...               \]

\vspace{5mm}
\noindent
and

\[ ...0 \xrightarrow{d_{-}(n)} 0 \xrightarrow{d_{-}(n)} H_{H}^{i_{0},j_{0}}(S) \xrightarrow{d_{-}(n)} H_{H}^{i_{0}+2n,j_{0}-2}(S)      \xrightarrow{d_{-}(n)}  ...               \]

\noindent

Since both complexes are bounded in horizontal grading, the two chain complexes are both finitely generated. They therefore have a well-defined Euler characteristic, and since the homologies are isomorphic, the Euler characteristics must be the same. Since the alternating sums of the dimension of homology is the same as the alternating sum of the dimension of the chain groups themselves, we have the following:

\[ \sum_{k=0}^{N} (-1)^{k} dim(H^{i_{0}+2nk, j_{0}-2k}(S)) = \sum_{k=0}^{N} (-1)^{k} dim(H_{H}^{i_{0}+2nk, j_{0}-2k}(S)) \]

Furthermore, we know that for $j<j_{0}$, there is an equality $dim(H^{i,j}(S)) = dim(H_{H}^{i,j}(S))$. Thus, the two  subcomplexes above have the same dimension in all of the bigradings except bigrading $(i_{0}, j_{0})$. But this means that for $k\ge 1$, the terms in the two sums are equal, which forces the $k=0$ terms to be equal. This contradicts our assumption that $dim(H^{i_{0},j_{0}}(S)) \ne dim(H_{H}^{i_{0},j_{0}}(S))$, proving the isomorphism.
\end{proof}

\begin{cor} \label{cor2.3}

Let $H(S)$ be a bigraded vector space that satisfies conditions 1-4 in Lemma \ref{thelemma} with one difference - instead of $d_{n}$ having bigrading $\{2n,-2\}$, it has bigrading $\{2n-2k,-2\}$. Then $H(S) \cong H_{H}(S)<k>$.

\end{cor}

\begin{proof}

After the change in grading $(gr_{q}, gr_{h}) \mapsto (gr_{q}+kgr_{h}, gr_{h})$, the conditions of Lemma \ref{thelemma} are satisfied, so this homology is isomorphic to $H_{H}(S)$. Shifting back to the original gradings proves the corollary. 
\end{proof}

To prove our main theorem, we now need to show that our homology 

\[ \bigoplus_{f} H_{1}(S_{f,1}) \otimes H_{H}(S_{f,2})\{  T_{2}(S_{f,1}) - T_{1}(S_{f,1}) - 2r(S_{f,2}),  0 \} \]

\noindent
satisfies the conditions of this corollary for $k=1$. It clearly satisfies conditions 1, 2, and 3 from Lemma \ref{thelemma} since each of the summands do, so we just need to define the $d_{n} $ differentials.

We will define $d_{n}$  for $n \ge 2$ as follows. It will preserve the direct sum decomposition, and it will act on each $H_{1}(S_{f,1}) \otimes H_{H}(S_{f,2})$ summand by $1 \otimes d_{-}(n-1)$, where $d_{-}(n-1)$ is the standard $sl_{n-1}$ differential on the HOMFLY-PT homology of $S_{f,2}$.

Since $H_{*}(H_{H}(S_{f,2}),d_{-}(n-1)) \cong H_{n-1}(S_{f,2})$, the homology with respect to $d_{n}$ is 

\[ \bigoplus_{f} H_{1}(S_{f,1}) \otimes H_{n-1}(S_{f,2})\{  T_{2}(S_{f,1}) - T_{1}(S_{f,1})-2r(S_{f,2}),  0 \} \]

From Section \ref{bigradedcomp}, we know that this sum is isomorphic as a bigraded vector space to $H_{n}(S)$. The differential $d_{n}$ has bigrading $\{2n-2, -2\}$, so applying Corollary \ref{cor2.3}, this proves Theorem \ref{thm2.1}.

\section{The Knot Floer Complex at a Vertex in the Cube of Resolutions} \label{HFKChapter}

\subsection{Definition of the Complex}

We will assume that the reader is familiar with Heegaard diagrams and knot Floer homology. For background on the subject, refer to \cite{OS1}, \cite{OS3}, \cite{Rasmussen2}. The oriented cube of resolutions for $\mathit{HFK}$ was originally defined with twisted coefficients by Ozsv\'{a}th and Szab\'{o}, and they noted some similarities between their complex and HOMFLY-PT homology (\hspace{1sp}\cite{Szabo}). The complex was further studied by Gilmore, who reframed the relationships in terms of framed trivalent graphs (\hspace{1sp}\cite{Gilmore}).

However, we will be dealing with the untwisted version defined by Manolescu in \cite{Manolescu}. This is in some ways the most natural version, as it doesn't involve twisted coefficients, and the total homology of the complex is the usual knot Floer homology. The oriented cube of resolutions uses the Heegaard diagram shown in Figure \ref{HeegaardDiagram}.

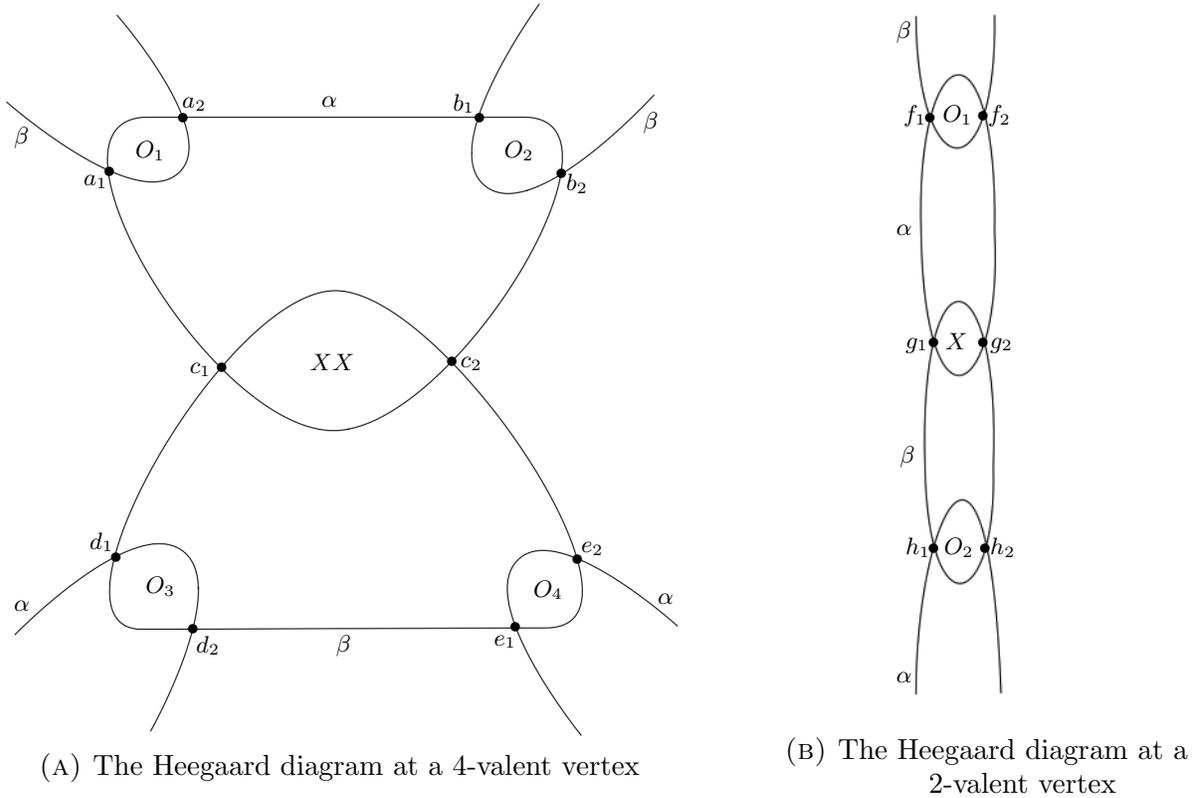
\begin{figure}[h!]
\begin{subfigure}{.6\textwidth}
  \centering
   \begin{overpic}[width=.9\textwidth]{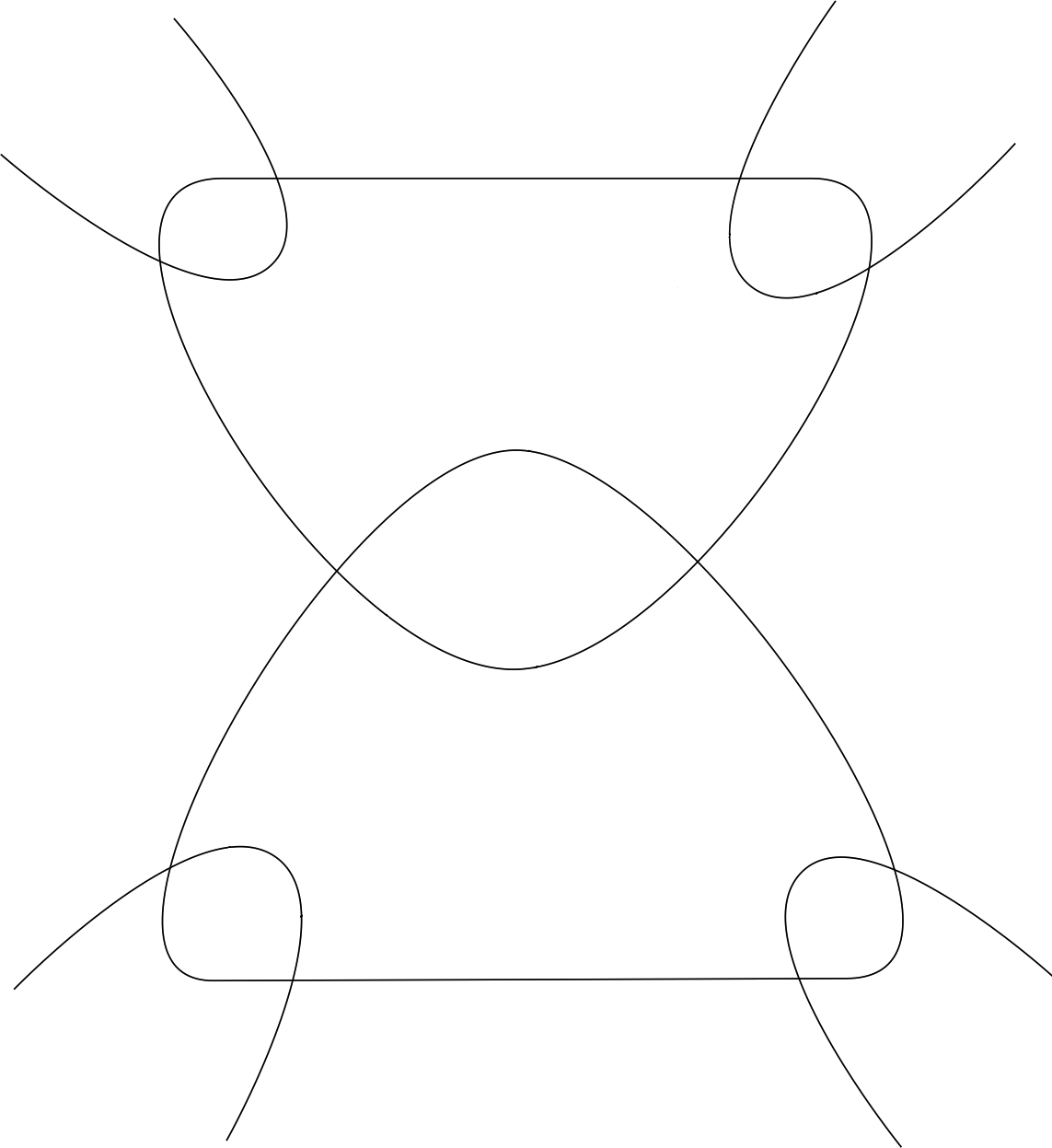} \scriptsize
   \put(17.6,79){$O_{1}$}
   \put(68,79){$O_{2}$}
   \put(19,19.5){$O_{3}$}
   \put(72,19){$O_{4}$}
   \put(41.5,50){$XX$}
   \put(43,86){$\alpha$}
   \put(1,17){$\alpha$}
   \put(89,18){$\alpha$}
   \put(1,81){$\beta$}
   \put(87,83){$\beta$}
   \put(45,11.5){$\beta$}
   \put(13.2,76.4){$\bullet$}
   \put(23.25,83.7){$\bullet$}
   \put(24,85.8){$a_{2}$}
   \put(10.5,75.2){$a_{1}$}
   \put(63.8,83.7){$\bullet$}
   \put(75,76){$\bullet$}
   \put(61,85.5){$b_{1}$}
   \put(76.5,74.5){$b_{2}$}
   \put(28.55,49.55){$\bullet$}
   \put(25,49.55){$c_{1}$}
   \put(60,50.3){$\bullet$}
   \put(62,50.3){$c_{2}$}
   \put(14.1,23.5){$\bullet$}
   \put(11.35,25.5){$d_{1}$}
   \put(24.6,13.8){$\bullet$}
   \put(25.9,11.34){$d_{2}$}
   \put(68.7,14){$\bullet$}
   \put(66.6,12.3){$e_{1}$}
   \put(77.15,23.3){$\bullet$}
   \put(78.5,25){$e_{2}$}
   \end{overpic}
  \caption{The Heegaard diagram at a 4-valent vertex}
  \label{fig2a}
\end{subfigure}%
\begin{subfigure}{.5\textwidth}
  \centering
  \hspace{-10mm}
  \begin{overpic}[width=.53\textwidth]{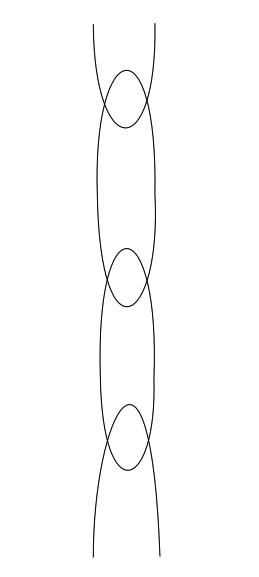} \scriptsize
   \put(19.8,81.5){$O_{1}$}
   \put(14.5,81.5){$f_{1}$}
   \put(26,81.5){$f_{2}$}
   \put(17.25,81.25){$\bullet$}
   \put(24.5,81.5){$\bullet$}
   \put(20,22.5){$O_{2}$}
   \put(14.75,22.5){$h_{1}$}
   \put(26.3,22.5){$h_{2}$}   
   \put(17.75,22.5){$\bullet$}
   \put(24.75,22.5){$\bullet$}
   \put(20.2,50.5){$X$}
   \put(14.75,50.65){$g_{1}$}
   \put(26.3,50.65){$g_{2}$}
   \put(17.75,50.65){$\bullet$}
   \put(24.5,50.65){$\bullet$}
   \put(13.5,66){$\alpha$}
   \put(13.5,93){$\beta$}
   \put(14,35){$\beta$}
   \put(13.5,5){$\alpha$}
  \end{overpic}
  \hspace{7mm} \caption{The Heegaard diagram at a \newline 2-valent vertex}
  \label{fig:sub2}
\end{subfigure}
\caption{The local Heegaard diagram for a singular link}
\label{HeegaardDiagram}
\end{figure}

Note that unlike in \cite{Manolescu}, we do not have a marked edge at which an $\alpha$ and a $\beta$ circle are removed. Instead, we place an additional $X$ and $O$ outside of our braid - we will denote this special $X$ by $X_{0}$. Since discs are not allowed to pass through $X_{0}$, this can be viewed as puncturing the sphere, making our diagram a truly planar diagram. We will also set the $U$ corresponding to the $O$ equal to zero to avoid increasing the ground ring. The net effect of this change is that we have added an unlinked component and reduced it, so the homology is twice that of \cite{Manolescu}.

The knot Floer complex corresponding to this Heegaard diagram is denoted \linebreak $\mathit{CFK^{-}}(S)$. There are several versions of knot Floer homology, and the minus sign refers to the fact that none of the $U_{i}$ in the ground ring will be set to $0$.

The complex ascribed to a vertex in the cube of resolutions, which we will denote $C_{F}(S)$, is the tensor product of $\mathit{CFK^{-}}(S)$ with a certain Koszul complex. Using the terminology from the previous section, we can define $C_{F}(S)$ as follows:

\[ C_{F}(S) = CFK^{-}(S) \otimes \bigotimes_{v \in V_{4}(S)} \Big( R \xrightarrow{\hspace{5mm}L(v)\hspace{5mm}} R \Big)  \]

\noindent
We will denote the Koszul complex by $K(S)$.

\subsubsection{The Generators of $CFK^{-}(S)$}

In order to understand the homology of $C_{F}(S)$, we are going to need some tools for understanding $\mathit{CFK}^{-}(S)$. Let $E(S)$ denote the set of edges of $S$, and let $x$ be a generator of the complex $\mathit{CFK}^{-}(S)$ (i.e. an n-tuple of intersection points of the $\alpha$ and $\beta$ curves). We ascribe a subset $Z$ of $E$ to the generator $x$ as follows. 

Each $O_{i}$ in the Heegaard diagram is contained in a unique minimal bigon. The boundary of this bigon contains two intersection points - if either of these intersection points are in the $n$-tuple $x$, then $e_{i}$ is in $Z$. For example, in Figure \ref{fig2a}, there are 5 types of generators: $(a,d)$, $(a,e)$, $(b,d)$, $(b,e)$, and $(c)$. The underlying sets of edges of these generators are $e_{1}e_{3}$, $e_{1}e_{4}$, $e_{2}e_{3}$, $e_{2}e_{4}$, and $\emptyset$, respectively.

As observed in \cite{Szabo}, $Z$ must satisfy two conditions. First, for any vertex $v$ in $S$, the number of incoming edges in $Z$ must equal the number of outgoing edges in $Z$, and second, $Z$ can not contain all four edges at any 4-valent vertex in $S$. In other words, $Z$ must be a disjoint union of oriented circles contained in $S$. We call such a set of edges a \emph{multi-cycle}. Note that multi-cycles differ from the homological cycles in Section \ref{clearcycle} in that multi-cycles cannot contain all four edges at a vertex.

Let $\mathit{CFK}^{-}(Z)$ denote the $R$-module spanned by generators $x$ where the multi-cycle underlying $x$ is $Z$, and let $C_{F}(Z)$ be the tensor product of $\mathit{CFK}^{-}(Z)$ with the Kozsul complex.

\[ C_{F}(Z) = \mathit{CFK}^{-}(Z) \otimes  \bigotimes_{v \in V_{4}(S)} \Big( R \xrightarrow{\hspace{5mm}L(v)\hspace{5mm}} R \Big)   \]

\subsection{The Filtered Complex and the Spectral Sequence from HOMFLY-PT Homology to HFK}

\subsubsection{A Filtration on $\mathit{CFK}^{-}(S)$} \label{filtration}

It turns out that there is a filtration on $\mathit{CFK}^{-}(S)$ that divides generators according to their underlying cycles. In other words, if there is a filtration-preserving differential $d$ with $d(x)=y$, then $x$ and $y$ have the same underlying cycle.

 This filtration is induced by placing additional basepoints $p_{i}$ in our Heegaard diagram as shown in Figure 3. The markings $p_{i}$ are in canonical bijection with regions in $\R^{2}-S$.

\begin{figure}[h!]
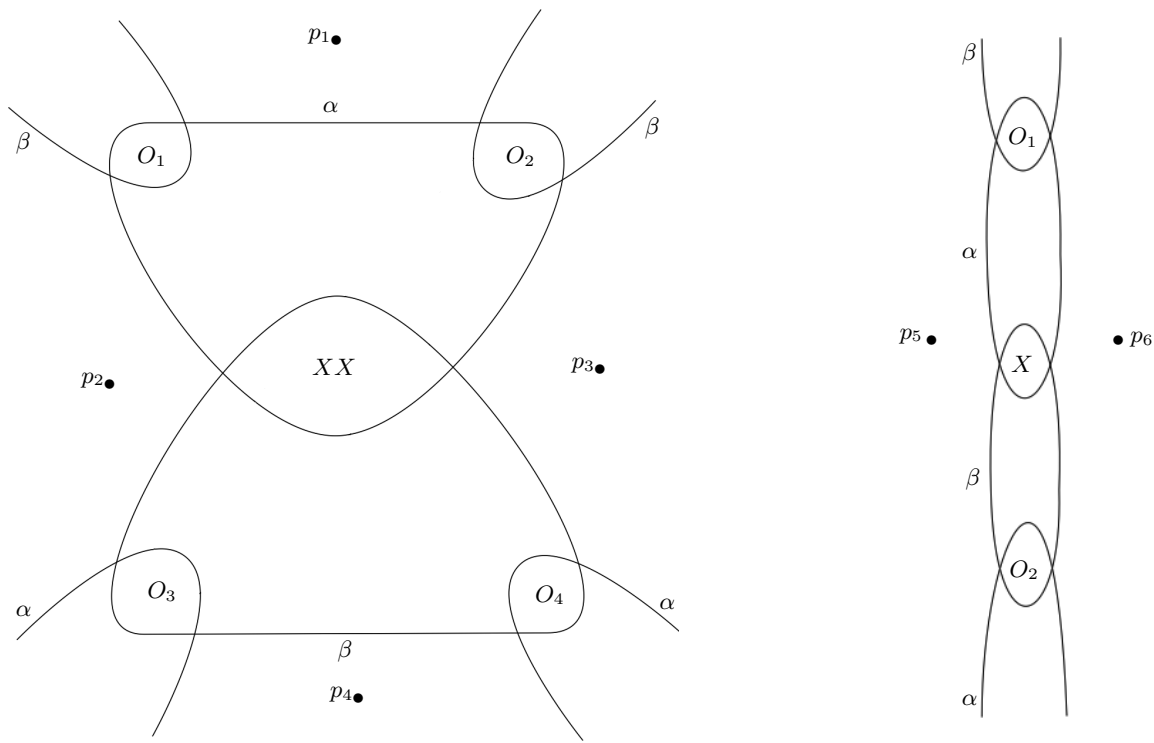

\centering
\begin{subfigure}{.6\textwidth}
  \centering
   \begin{overpic}[width=.9\textwidth]{initial_diagram.png} \scriptsize
   \put(17.6,79){$O_{1}$}
   \put(68,79){$O_{2}$}
   \put(19,19.5){$O_{3}$}
   \put(72,19){$O_{4}$}
   \put(41.5,50){$XX$}
   \put(43,86){$\alpha$}
   \put(1,17){$\alpha$}
   \put(89,18){$\alpha$}
   \put(1,81){$\beta$}
   \put(87,83){$\beta$}
   \put(45,11.5){$\beta$}
   \put(13,48){$\bullet$}
   \put(80,50){$\bullet$}
   \put(47,5){$\bullet$}
   \put(44,95){$\bullet$}
   \put(10,49){$p_{2}$}
   \put(77,51){$p_{3}$}
   \put(44,6){$p_{4}$}   
   \put(41,96){$p_{1}$}
   \end{overpic}
  \label{fig:sub1}
\end{subfigure}%
\begin{subfigure}{.5\textwidth}
  \centering
  \begin{overpic}[width=.53\textwidth]{MarkedEdgeHD.png} \scriptsize
   \put(19.8,81.5){$O_{1}$}
   \put(20,22.5){$O_{2}$}
   \put(20.2,50.5){$X$}
   \put(5,55){$p_{5}$}
   \put(36.5,54.5){$p_{6}$}
   \put(8.5,54){$\bullet$}
   \put(34,54){$\bullet$}
   \put(13.5,66){$\alpha$}
   \put(13.5,93){$\beta$}
   \put(14,35){$\beta$}
   \put(13.5,5){$\alpha$}
  \end{overpic}
  \label{fig:sub2}
\end{subfigure}
\caption{Local Diagrams with Additional Markings}
\label{fig:test}
\end{figure}

\begin{lem} \label{lemma4.1}

These markings define a filtration on the complex $\mathit{CFK}^{-}(S)$, where the change in filtration level of a differential is given by the sum of the multiplicities of the corresponding holomorphic disc at these basepoints. 

\end{lem}

\begin{proof}

It is sufficient to show that any periodic domain has multiplicity zero at these markings. This follows from that fact that for any $\alpha$ or $\beta$ circle, the markings and $X_{0}$ lie on the same side. So for any periodic domain, the multiplicity at any of these points is the same as that of $X_{0}$, which is required to be zero.
\end{proof}

We extend this filtration to $C_{F}(S)$ by placing the Koszul complex in a single filtration level. Let $d_{k}$ denote the component of the differential on $C_{F}(S)$ which increases the filtration by $k$.

\begin{lem}

The differential $d_{0}$ preserves $C_{F}(Z)$, i.e. it does not change the underlying cycle of a generator.

\end{lem}

\begin{proof}

Each basepoint gives a filtration on our complex corresponding to a region in the knot projection. Let $x$ be a generator with multi-cycle $Z$, and let $C$ be an oriented 2-chain with boundary $Z$. If we require that $C$ has multiplicity $0$ on the outer region (the one corresponding to $X_{0}$), it is clear that this 2-chain is unique.

Within the planar Heegaard diagram for $K$, we can find a disc that connects $x$ to a generator corresponding to the empty cycle (see \cite{Szabo}, section 3), and since each region in the knot projection contains a basepoint, the multiplicities of the disc at each basepoint will be equal to the multiplicity of $C$ in that region. Thus, each filtration level uniquely determines a 2-chain $C$, whose boundary gives the multi-cycle $Z$. Since no two multi-cycles correspond to the same 2-chain, this completes the proof.
\end{proof}

This means that the filtered homology, also called the homology of the associated graded object, must split over the multi-cycles $Z$. 

\subsubsection{Homology of a Cycle}

Before computing the homology $H(C_{F}(Z),d_{0})$, we will need a definition. If $S$ is a singular braid and $Z$ is a multi-cycle in $S$, let $S-Z$ denote the diagram obtained by removing all edges in $Z$ from $S$. Note that $S-Z$ is still a singular braid because $Z$ is an oriented cycle in the graph.

Given a cycle $Z$, the complex $\mathit{CFK}^{-}(Z)$ is easy to compute. Each intersection point in the Heegaard diagram lies on a unique convex bigon (convex in the traditional planar geometry sense), and this bigon either contains an $X$, an $XX$, or a $O_{i}$. There are canonical bijections between the $O_{i}$ bigons and the edges $e_{i}$, between the $X$ bigons and $V_{2}(S)$ and between the $XX$ bigons and $V_{4}(S)$.

Given a generator $x$, let $W_{2}(x)$ denote the set of vertices at which $x$ has an intersection point on one of the $X$ bigons, and let $W_{4}(x)$ denote the set of vertices at which $x$ has an intersection point on one of the $XX$ bigons. Note that $W_{2}(x)$ and $W_{4}(x)$ are uniquely determined by the underlying cycle $Z$ of $x$. In particular, $W_{2}(x)$ and $W_{4}(x)$ are those vertices which are not endpoints of any edges in $Z$. We can therefore define $W_{2}(Z)$ and $W_{4}(Z)$ accordingly.

The complex for a cycle $Z$ can now be described as follows. Each edge $e_{i}$ in $Z$ corresponds to two intersection points, which are connected by a bigon containing $O_{i}$. These are the only filtered differentials involving these two intersection points, so $\mathit{CFK}^{-}(Z)$ is going to come with a tensor summand of the Koszul complex

\[ \bigotimes_{e_{i} \in Z}  R \xrightarrow{\hspace{5mm}U_{i}\hspace{5mm}} R   \]

Each vertex $v$ in $W_{2}(Z)$ also corresponds to two intersection points. They are connected by two bigons, one which passes through $O_{i}$ (where $e_{i}$ is the outgoing edge from $v$) and one which passes through $O_{j}$ (where $e_{j}$ is the incoming edge at $v$). These two bigons will give a coefficient of $\pm(U_{i}-U_{j})$. Thus, we also get a tensor summand of the Koszul complex

\[ \bigotimes_{v \in W_{2}(Z)}  R \xrightarrow{\hspace{5mm}L(v)\hspace{5mm}} R   \]

\noindent
Proving the signs requires slightly more advance machinery, and will be discussed at the end of the section.

Finally, the vertices $v$ in $W_{4}(Z)$ correspond to two intersection points, also connected by two bigons. One passes through $O_{i}$ and $O_{j}$, where $e_{i}$ and $e_{j}$ are the outgoing edges of $v$, and the other passes through $O_{k}$ and $O_{l}$, where $e_{k}$ and $e_{l}$ are the incoming edges at $v$. These two bigons will contribute a coefficient of $\pm(U_{i}U_{j}-U_{k}U_{l})$, giving us the last Koszul complex

\[\bigotimes_{v \in W_{4}(Z)} R \xrightarrow{\hspace{5mm}Q(v)\hspace{5mm}} R   \]

\noindent
These are all the generators and all the differentials, so the total complex is given by 

\[ \mathit{CFK}^{-}(Z) = \Big[ \bigotimes_{e_{i} \in Z}  R \xrightarrow{\hspace{3mm}U_{i}\hspace{3mm}} R \Big] \otimes \Big[ \bigotimes_{v \in W_{2}(Z)} R \xrightarrow{\hspace{3mm}L(v)\hspace{3mm}} R \Big] \otimes \Big[ \bigotimes_{v \in W_{4}(Z)} R \xrightarrow{\hspace{3mm}Q(v)\hspace{3mm}} R \Big] \]

\noindent
and so the total complex for $C_{F}(Z)$ is given by

\[ C_{F}(Z) = \Big[ \bigotimes_{e_{i} \in Z} R \xrightarrow{\hspace{0mm}U_{i}\hspace{0mm}} R \Big] \otimes \Big[ \bigotimes_{v \in W_{2}(Z)} R \xrightarrow{\hspace{-.25mm}L(v)\hspace{-.25mm}} R \Big] \otimes \Big[ \bigotimes_{v \in W_{4}(Z)} R \xrightarrow{\hspace{-.25mm}Q(v)\hspace{-.25mm}} R \Big] \otimes \Big[ \bigotimes_{v \in V_{4}(S)} R \xrightarrow{\hspace{-.25mm}L(v)\hspace{-.25mm}} R \Big]  \]

\begin{lem}\label{bigthm}

The filtered homology $H_{*}(C_{F}(Z),d_{0})$ is isomorphic to $H_{H}(S-Z)$.

\end{lem}

\begin{proof}

The $U_{i}$ in the first tensor product form a regular sequence in $R$, so we can cancel all of these differentials. This has the effect of setting $U_{i}$ equal to zero for all $e_{i}$ in $Z$. Let $R_{Z}$ be the quotient $R/\{U_{i}=0 \text{ for } e_{i} \in Z\}$. Note that this is precisely the ground ring for the singular braid $S-Z$.

We are left with the complex

\[\Big[ \bigotimes_{v \in W_{2}(Z)} R_{Z} \xrightarrow{\hspace{1mm}L(v)\hspace{1mm}} R_{Z} \Big] \otimes \Big[ \bigotimes_{v \in W_{4}(Z)} R_{Z} \xrightarrow{\hspace{1mm}Q(v)\hspace{1mm}} R_{Z} \Big] \otimes \Big[ \bigotimes_{v \in V_{4}(S)} R_{Z} \xrightarrow{\hspace{1mm}L(v)\hspace{1mm}} R_{Z} \Big]  \]

For each 4-valent vertex $v$ in $S-Z$, we have tensor summands $R_{Z} \xrightarrow{\hspace{5mm}L(v)\hspace{5mm}} R_{Z}$ and $R_{Z} \xrightarrow{\hspace{5mm}Q(v)\hspace{5mm}} R_{Z}$, which together give a summand of

\begin{figure}[!h]
\centering
\begin{tikzpicture}
  \matrix (m) [matrix of math nodes,row sep=5em,column sep=6em,minimum width=2em] {
     R_{Z} & R_{Z} \\
     R_{Z} & R_{Z} \\};
  \path[-stealth]
    (m-1-1) edge node [left] {$Q(v)$} (m-2-1)
    (m-1-1) edge node [above] {$L(v)$} (m-1-2)
    (m-1-2) edge node [right] {$Q(v)$} (m-2-2)
    (m-2-1) edge node [above] {$L(v)$} (m-2-2);
\end{tikzpicture}
\end{figure}

\noindent
which is precisely the HOMFLY-PT summand $C_{H}(v)$. For 2-valent vertices $v$ in $S-Z$, there are two possibilities to consider - $v$ is 2-valent in $S$ ($v \in W_{2}(Z))$, and $v$ is 4-valent in $S$ ($v \in V_{4}(S), v \notin W_{4}(Z)$). When $v$ is 2-valent in $S$, we get the summand 

\[R_{Z} \xrightarrow{\hspace{5mm}L(v)\hspace{5mm}} R_{Z} \]

\noindent
which is again the HOMFLY-PT summand $C_{H}(v)$ for $S-Z$. For $v$ 4-valent in $S$, let $e_{i}$ and $e_{j}$ be the outgoing edges at $v$ and $e_{k}$ and $e_{l}$ the incoming edges at $v$. Since $S-Z$ is 2-valent at $v$, we know that $Z$ must include one outgoing edge and one incoming edge. Without loss of generality, assume they are $e_{i}$ and $e_{k}$. To avoid confusion, we will write out the terms of the linear elements, as $L(v)$ refers to $U_{i}+U_{j}-U_{k}-U_{l}$ in $S$, while $L(v)$ refers to $U_{j}-U_{l}$ in $S-Z$.

In $C_{F}(Z)$, we have the summand 

\[R_{Z} \xrightarrow{\hspace{5mm}U_{i}+U_{j}-U_{k}-U_{l}\hspace{5mm}} R_{Z} \]

\noindent
In the HOMFLY-PT complex for $S-Z$, on the other hand, we have the summand

\[R_{Z} \xrightarrow{\hspace{5mm}U_{j}-U_{l}\hspace{5mm}} R_{Z} \]

\noindent
Fortunately, since $e_{i}$ and $e_{k}$ are in $Z$, $U_{i}$ and $U_{k}$ are zero in $R_{Z}$, so $U_{i}+U_{j}-U_{k}-U_{l}=U_{j}-U_{l}$, making the above complexes isomorphic. 

Thus, after canceling the Koszul complex on the edges in $Z$, we get exactly the HOMFLY-PT complex for $S-Z$. It follows that $H_{*}(C_{F}(Z),d_{0}) \cong H_{H}(S-Z)$.
\end{proof}

\begin{cor}\label{cor1} 
The filtered homology decomposes as the direct sum 

\[H_{*}(C_{F}(S), d_{0}) \cong  \bigoplus_{Z} H_{H}(S-Z)\]
\end{cor}

\begin{rem}\textbf{Signs}. Since we are discussing Koszul complexes, the $\pm$ in the terms $\pm(U_{i}-U_{j})$ and $\pm(U_{i}U_{j}-U_{k}U_{l})$ are not relevant. Some will have to come with positive signs and some with negative to make $d^{2}=0$, but where they are doesn't impact the chain homotopy type. What we need to show is that the two bigons in each case come with \emph{different} signs.

The two-valent vertex corresponds to a specific $X$ marking in the diagram. This $X$ lies within the same $\alpha$ circle as $O_{i}$ and the same $\beta$ circle as $O_{j}$. In $\mathit{CFK}^{-}$, we do not allow discs to pass through the $X$ basepoints. However, if we do allow them to pass through only this $X$, we get a new complex. In this complex, $d^{2}$ is non-zero - instead, it is a multiple of the identity. This multiple is determined by the $\alpha$ and $\beta$ degenerations, which will correspond to the $\alpha$ and $\beta$ circles containing $X$. Since the $\alpha$ circle contains $O_{i}$, it gives a coefficient of $U_{i}$, and similarly, the $\beta$ circle gives a coefficient of $U_{j}$. It was shown in \cite{Akram} that the moduli spaces can be oriented such that the $\alpha$ and $\beta$ degenerations come with opposite signs, so choosing such a sign convention, this gives 

\[   d^{2}=\pm(U_{i}-U_{j}) I  \]

\noindent
Moreover, the additional differentials are also subject to the basepoint filtration, so we get 

\[   d_{0}^{2}=\pm(U_{i}-U_{j}) I  \]

This $X$ basepoint lies inside a minimal bigon, and this bigon now contributes to the differential with a coefficient of $\pm1$. The local contribution therefore must be the the complex in Figure \ref{fig4.3}, so $U_{i}$ and $U_{j}$ must come with opposite sign.

\begin{figure}[!h]
\centering
\begin{tikzpicture}
  \matrix (m) [matrix of math nodes,row sep=5em,column sep=8em,minimum width=2em] {
     R & R \\};
  \path[-stealth]
    (m-1-1) edge [bend left=15] node [above] {$\pm(U_{i}-U_{j})$} (m-1-2)
    (m-1-2) edge [bend left=15] node [below] {$\pm1$} (m-1-1);
\end{tikzpicture}
\caption{Local complex when allowing discs to pass through the $X$ basepoint.}
\label{fig4.3}
\end{figure}

The argument for the quadratic term is the same, only instead of allowing discs to pass through an $X$, we are allowing them to pass through an $XX$. The $\alpha$ degeneration is $U_{i}U_{j}$ and the $\beta$ degeneration is $U_{k}U_{l}$, and they must come with opposite sign, so we get the complex in Figure \ref{quadcomplex}, which proves that $U_{i}U_{j}$ and $U_{k}U_{l}$ come with opposite sign.

\begin{figure}[!h]
\centering
\begin{tikzpicture}
  \matrix (m) [matrix of math nodes,row sep=5em,column sep=8em,minimum width=2em] {
     R & R \\};
  \path[-stealth]
    (m-1-1) edge [bend left=15] node [above] {$\pm(U_{i}U_{j}-U_{k}U_{l})$} (m-1-2)
    (m-1-2) edge [bend left=15] node [below] {$\pm1$} (m-1-1);
\end{tikzpicture}
\caption{Local complex when allowing discs to pass through the $XX$ basepoint.}
\label{quadcomplex}
\end{figure}

\end{rem}

\subsubsection{Gradings} \label{gradings}

The knot Floer complex comes equipped with two gradings: the Maslov grading $M$ and the Alexander grading $A$. The differential decreases the Maslov grading by 1 and preserves the Alexander grading. Multiplication by $U_{i}$ decreases the Maslov grading by 2 and decreases the Alexander grading by 1.

Certain linear combinations of the Maslov and Alexander gradings return analogs of the quantum and horizontal gradings from the Khovanov-Rozansky complex. Let $gr_{q}$ be given by $-2M+2A$, and $gr_{h}$ by $-2M+4A$. Note that the knot Floer differential has bigrading $\{2,2\}$ with respect to this differential and multiplication by $U_{i}$ changes the bigrading by $\{2,0\}$, the same as the Khovanov-Rozansky complex. Instead of the Maslov and Alexander gradings, we will henceforth use the quantum and horizontal gradings.

Before computing gradings, we need to introduce some terminology. Given a multi-cycle $Z$, let $T_{1}(Z)$ denote the the number of vertices $v \in V_{4}(S)$ at which $Z$ contains the edges $e_{1}$ and $e_{3}$ in Figure \ref{vertex2}. Similarly, let $D_{1}(Z)$ denote the number of vertices at which $Z$ contains the edges $e_{1}$ and $e_{4}$, $D_{2}(Z)$ the number of vertices at which $Z$ contains the edges $e_{2}$ and $e_{3}$, and $T_{2}(Z)$ denote the the number of vertices at which $Z$ contains the edges $e_{2}$ and $e_{4}$. 

\begin{figure}[h!]
 \centering
   \begin{overpic}[width=.3\textwidth]{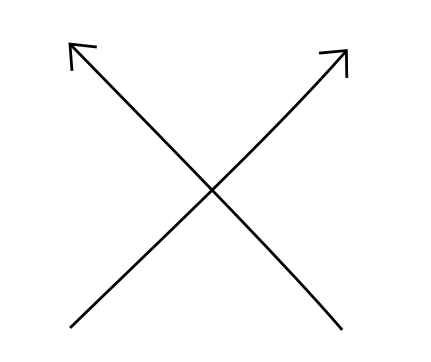}
   \put(23,54){$e_{1}$}
   \put(68,54){$e_{2}$}
   \put(23,23){$e_{3}$}
   \put(68,23){$e_{4}$}
   \end{overpic}
\caption{A labeled 4-valent vertex}
\label{vertex2}
\end{figure}

We will now compute the bigrading on the knot Floer complex, up to an overall grading shift. Since the generators corresponding to the empty cycle give a complex which is isomorphic to the HOMFLY-PT complex for $S$, we can choose our overall shift so that this isomorphism preserves the bigrading. 

Let $Z$ denote a multi-cycle in $S$ consisting of $k$ cycles. Viewing $Z$ as a braid diagram for the $k$-component unlink, we can define $r(Z)$ to be the rotation number of $Z$. Thus, $r(Z)=-k$. Let $x$ denote the generator corresponding to $Z$ at the bottom of the Koszul complex (i.e. with the largest horizontal grading). Similarly, let $y$ denote the generator corresponding to the empty cycle with the largest horizontal grading. In \cite{Szabo}, Sz\'{a}bo and Ozsv\'{a}th identify $k$ differentials whose composition takes $x$ to $y$, whose coefficient in $R$ has degree $T_{2}(Z)+\frac{1}{2}(D_{1}(Z)+D_{2}(Z))$. Using the fact that differentials have bigrading $\{2, 2\}$, we can see that $x$ and $y$ differ in grading by 

\[ \{2r(Z) + 2T_{2}(Z) + D_{1}(Z) + D_{2}(Z)), 2r(Z)\} \]

The bottom generator of the HOMFLY-PT complex has bigrading $\{-|V_{4}(S)|,0\}$, so $y$ does as well. Thus, $x$ has bigrading 

\[ \{-|V_{4}(S)| + 2r(Z) + 2T_{2}(Z) + D_{1}(Z)+D_{2}(Z), 2r(Z)\} \]

\noindent
The bottom generator of the HOMFLY-PT complex for $S-Z$ has bigrading 

\noindent
$\{-|V_{4}(S-Z)|,0\}$, so we get the following:

\[ H(C_{F}(Z),d_{0}) \cong H_{H}(S-Z)\{|V_{4}(S-Z)|-|V_{4}(S)| + 2r(Z) + 2T_{2}(Z) + D_{1}(Z)+D_{2}(Z), 2r(Z)\} \]

The grading shift in this formula can be simplified somewhat. The quantity $|V_{4}(S-Z)|-|V_{4}(S)|$ is the negative of the number of 4-valent vertices in $S$ at which $Z$ contains two edges. 

\[ |V_{4}(S-Z)|-|V_{4}(S)| = -T_{1}(Z)-T_{2}(Z)-D_{1}(Z)-D_{2}(Z) \]

\noindent
Thus, the formula becomes 

\[ H(C_{F}(Z),d_{0}) \cong H_{H}(S-Z)\{ 2r(Z) + T_{2}(Z) - T_{1}(Z),  2r(Z)\} \]

\noindent
Our choice for the overall grading shift was somewhat arbitrary, so to make it more similar to our composition product formulas, we will add in a grading shift of $\{-2r(S), 0 \}$ giving

\[ H(C_{F}(Z),d_{0}) \cong H_{H}(S-Z)\{ -2r(S-Z) + T_{2}(Z) - T_{1}(Z),  2r(Z)\} \]

\noindent
and we get a graded version of Corollary \ref{cor1}:

\[H_{*}(C_{F}(S), d_{0}) \cong  \bigoplus_{Z} H_{H}(S-Z)\{  -2r(S-Z) + T_{2}(Z) - T_{1}(Z), 2r(Z)\}  \]

We will be able to connect this formula to the composition product with the following lemma:

\begin{lem} \label{sl1}

Let $f$ be a labeling of $S$. The $sl_{1}$ homology of $S_{f,1}$ is given by 

\[ H_{1}(S_{f,1}) = \begin{cases} 
      \Q \{0, 2r(S_{f,1})\} & \textrm{ if $S_{f,1}$ is a multi-cycle} \\
      0 & \textrm{ otherwise} \\ 
   \end{cases} \]

\end{lem}

Applying this lemma, the formula becomes
\[ H_{*}(C_{F}(S), d_{0}) \cong   \bigoplus_{f} H_{1}(S_{f,1}) \otimes H_{H}(S_{f,2})\{  T_{2}(S_{f,1}) - T_{1}(S_{f,1})-2r(S_{f,2}),  0 \} \]

\noindent
But by Theorem \ref{thm2.1}, this is isomorphic to $H_{H}(S)<1>$. Thus, we have proved the following theorem:

\begin{thm} \label{filterediso}

There is an isomorphism of bigraded groups 

\[H_{*}(C_{F}(S), d_{0}) \cong H_{H}(S)<1>\]

\end{thm}

\begin{cor}

There is a spectral sequence whose $E_{1}$ page is $H_{H}(S)<1>$ which converges to $H_{F}(S)$.

\end{cor}

\begin{proof}

This is just the spectral sequence induced by the basepoint filtration on $C_{F}(S)$.
\end{proof}

Manolescu's conjecture is thus equivalent to this spectral sequence collapsing at the $E_{1}$ page.

\subsection{Additional Differentials and the Spectral Sequences from HFK to $sl_{n}$}

We are going to add differentials to the complex $C_{F}(S)$ so that the total homology is isomorphic to $H_{n+1}(S)$ for any $n \ge 1$. These new differentials do not preserve the Alexander grading, so using the Alexander grading as a filtration, this induces a spectral sequence from $H_{F}(S)$ to $H_{n+1}(S)$.

The complex $C_{F}(S)$ is constructed as a tensor product of complexes $\mathit{CFK}^{-}(S)$ and a Koszul complex $K(S)$ on linear elements. 

\[ C_{F}(S) = \mathit{CFK}^{-}(S) \otimes K(S) \]

The complex $\mathit{CFK}^{-}(S)$ does not count discs which pass through the $X$ or $XX$ markings. For the new differential, we are going to count these discs with certain polynomial coefficients.

Each $X$ marking in the Heegaard diagram corresponds to a 2-valent vertex $v$ in $S$. Whenever a holomorphic disc passes through this $X$ with multiplicity $k$, it picks up a coefficient of $u_{1}(v)^{k}$. The only exception is the special marking $X_{0}$, at which we still require discs to have multiplicity 0. Similarly, each $XX$ corresponds to a 4-valent vertex $v$ in $S$. If a holomorphic disc passes through this $XX$ with multiplicity $k$, it picks up a coefficient of $u_{2}(v)^{k}$. We will call this new complex $\mathit{CFK}^{-}_{n}(S)$.

Note that there is no guarantee that the differential on this complex squares to zero - in fact, it doesn't. 

To fix this, we will also modify the differential on the Koszul complex. Originally, it was given by 

\[ K(S) = \bigotimes_{v \in V_{4}(S)} R \xrightarrow{\hspace{5mm}L(v)\hspace{5mm}} R  \]

\newcommand{\myrightleftarrows}{\mathrel{\substack{\xrightarrow{\hspace{5mm}L(v) \hspace{5mm}} \\[-.4ex] \xleftarrow[\hspace{5mm}u_{1}(v) \hspace{5mm}]{}}}}

\noindent
We are going to add in differentials to make it a matrix factorization:

\[
K_{n}(S) = \bigotimes_{v \in V_{4}(S)} R \mathrel{\substack{\\\myrightleftarrows\\}}  R
\]

\noindent
The total complex $C_{F(n)}(S)$ is defined to be the tensor product of $\mathit{CFK}^{-}_{n}(S)$ and $K_{n}(S)$.

\[ C_{F(n)}(S) = \mathit{CFK}^{-}_{n}(S) \otimes K_{n}(S) \]

\begin{lem}

The differential on $C_{F(n)}(S)$ satisfies $d^{2}=0$.

\end{lem}

\begin{proof}

At each vertex $v$ in $S$, $d^{2}$ is going have a local contribution of $w_{n}(v)$, which is given by  \[ \sum_{e_{i} \in E_{out}} U_{i}^{n+1} -  \sum_{e_{j} \in E_{in}} U_{j}^{n+1} \]

\noindent
The lemma will then follow from the equality $\sum_{v \in S} w_{n}(v) = 0$.

The quantity $d^{2}$ has two contributions, one from $\mathit{CFK}^{-}_{n}(S)$ and one from $K_{n}(S)$. The contribution from $K_{n}(S)$ can be computed directly to be 

\[ \sum_{v \in V_{4}(S)} L(v)u_{1}(v) \]

The contribution from $\mathit{CFK}^{-}_{n}(S)$ can be computed via the $\alpha$ and $\beta$ degenerations. We orient the moduli spaces so that the $\alpha$ and $\beta$ degenerations come with opposite signs, with the $\alpha$ degenerations being positive and the $\beta$ degenerations negative (see \cite{Akram} for an explanation of the signs).

At each 2-valent vertex $v$ in $S$, we have one $\alpha$ circle and one $\beta$ circle, shown in Figure \ref{BivalentDegenerations} . The $\alpha$ circle contains $U_{i}$ and $X$, and the $X$ contributes coefficient $u_{1}(v)$, so the $\alpha$ degeneration contributes $U_{i}u_{1}(v)$. Similarly, the $\beta$ circle contains $U_{j}$ and $X$, so its contribution is $-U_{j}u_{1}(v)$. Thus, the net contribution at $v$ is $(U_{i}-U_{j})u_{1}(v)$. This can be simplified to $L(v)u_{1}(v) = w_{n}(v)$.

\begin{figure}[h!]
\centering
\begin{subfigure}{.5\textwidth}
  \centering
   \begin{overpic}[width=.38\textwidth]{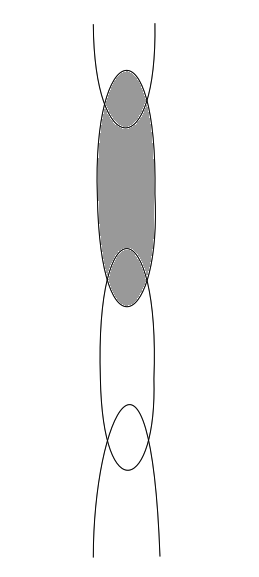} \tiny
   \put(19.8,81.5){$O_{i}$}
   \put(20,22.5){$O_{j}$}
   \put(20.2,50.5){$X$}
   \put(13.5,66){$\alpha$}
   \put(13.5,93){$\beta$}
   \put(14,35){$\beta$}
   \put(13.5,5){$\alpha$}
   \end{overpic}
  \label{fig:sub1}
\end{subfigure}%
\begin{subfigure}{.5\textwidth}
  \centering
  \begin{overpic}[width=.38\textwidth]{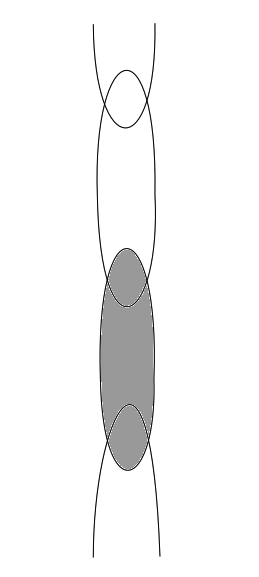} \tiny
   \put(19.8,81.5){$O_{i}$}
   \put(20,22.5){$O_{j}$}
   \put(20.2,50.5){$X$}
   \put(13.5,66){$\alpha$}
   \put(13.5,93){$\beta$}
   \put(14,35){$\beta$}
   \put(13.5,5){$\alpha$}
  \end{overpic}
  \label{fig:sub2}
\end{subfigure}
\caption{$\alpha$-degenrations (left) and $\beta$-degenerations (right) at a bivalent vertex}
\label{BivalentDegenerations}
\end{figure}
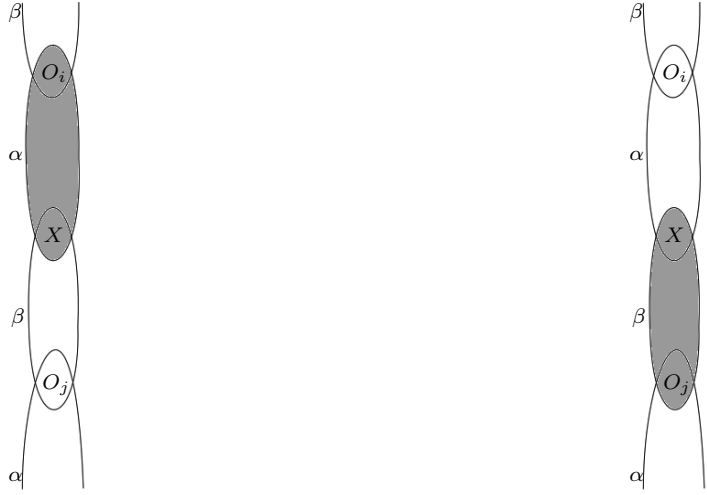

At each 4-valent vertex $v$ in $S$, we also have one $\alpha$ circle and one $\beta$ circle, shown in Figure \ref{4valentdegen}. The $\alpha$ circle contains $U_{i}, U_{j},$ and $XX$, and the $XX$ contributes coefficient $u_{2}(v)$, so its contribution is $U_{i}U_{j}u_{2}(v)$. The $\beta$ circle contains $U_{k}$, $U_{l}$, and $XX$, so its contribution is $-U_{k}U_{l}u_{2}(v)$. Thus, the net contribution at $v$ is $(U_{i}U_{j}-U_{k}U_{l})u_{2}(v)$. This can be simplified to $Q(v)u_{2}(v)$.

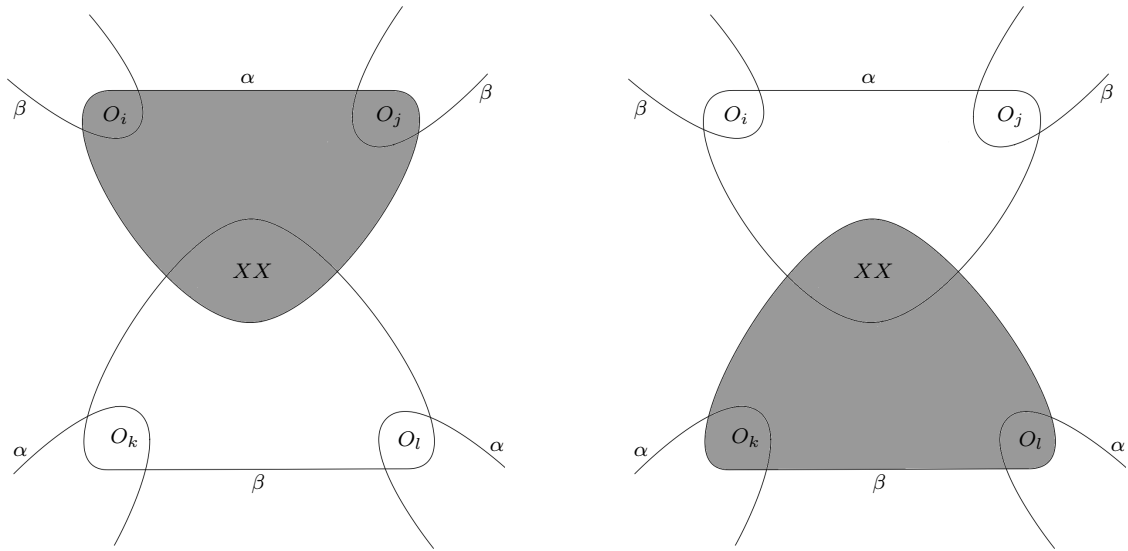
\begin{figure}[h!]
\centering
\begin{subfigure}{.5\textwidth}
  \centering
   \begin{overpic}[width=.8\textwidth]{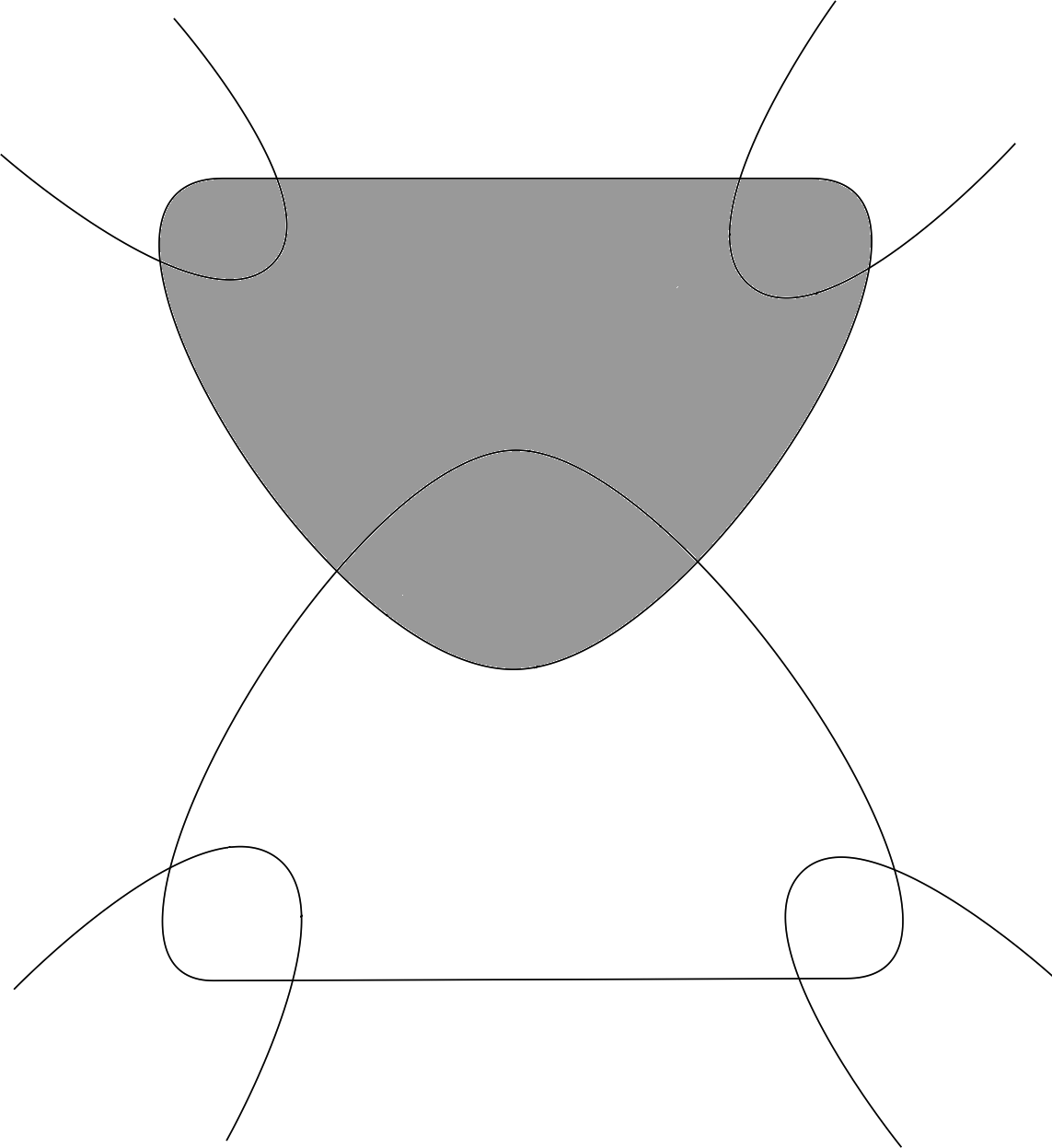} \tiny
   \put(17.6,79){$O_{i}$}
   \put(68,79){$O_{j}$}
   \put(19,19.5){$O_{k}$}
   \put(72,19){$O_{l}$}
   \put(41.5,50){$XX$}
   \put(43,86){$\alpha$}
   \put(1,17){$\alpha$}
   \put(89,18){$\alpha$}
   \put(1,80){$\beta$}
   \put(87,83){$\beta$}
   \put(45,11){$\beta$}
   \end{overpic}
  \label{fig:sub1}
\end{subfigure}%
\begin{subfigure}{.5\textwidth}
  \centering
   \begin{overpic}[width=.8\textwidth]{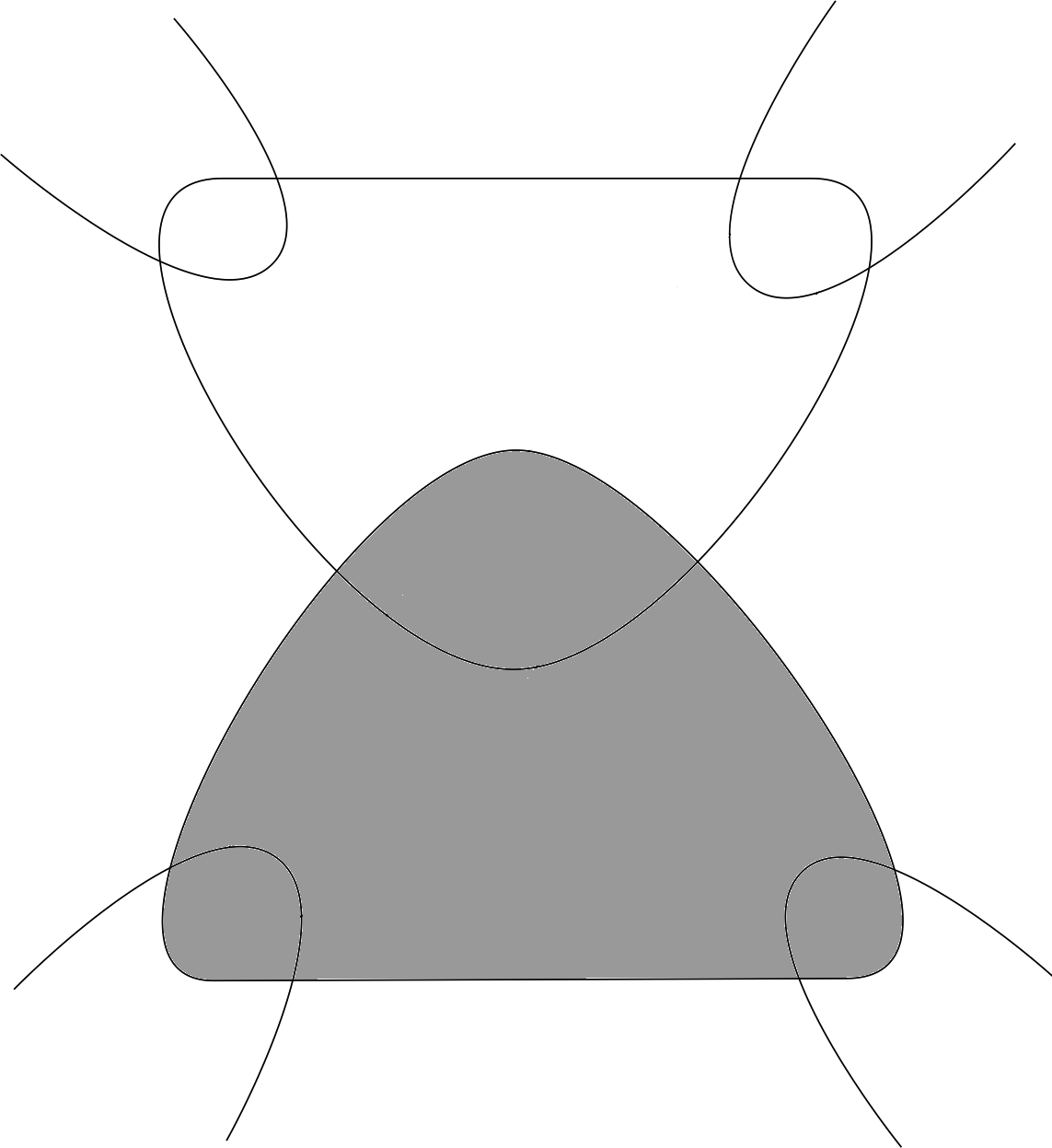} \tiny
   \put(17.6,79){$O_{i}$}
   \put(68,79){$O_{j}$}
   \put(19,19.5){$O_{k}$}
   \put(72,19){$O_{l}$}
   \put(41.5,50){$XX$}
   \put(43,86){$\alpha$}
   \put(1,17){$\alpha$}
   \put(89,18){$\alpha$}
   \put(1,80){$\beta$}
   \put(87,83){$\beta$}
   \put(45,11){$\beta$}
  \end{overpic}
  \label{fig:sub2}
\end{subfigure}
\caption{$\alpha$-degenrations (left) and $\beta$-degenerations (right) at a four-valent vertex}
\label{4valentdegen}
\end{figure}

Thus, counting the contribution from $K_{n}(S)$, we see that $d^{2}$ is given by

\begin{equation*}
\begin{split}
d^{2} =& \sum_{v \in V_{2}(S)} w_{n}(v) +  \sum_{v \in V_{4}(S)} Q(v)u_{2}(v) + \sum_{v \in V_{4}(S)} L(v)u_{1}(v) \\
=& \sum_{v \in V_{2}(S)} w_{n}(v) +  \sum_{v \in V_{4}(S)} L(v)u_{1}(v) + Q(v)u_{2}(v)\hspace{6mm} \\
=& \sum_{v \in V_{2}(S)} w_{n}(v) +  \sum_{v \in V_{4}(S)} w_{n}(v) \\
=& \sum_{v \in S} w_{n}(v) =0
\end{split}
\end{equation*}
\end{proof}

We can extend the basepoint filtration from Section \ref{filtration} to make $C_{F(n)}(S)$ a filtered complex - since we still require discs to have multiplicity 0 at $X_{0}$, the same argument works as in the proof of Lemma \ref{lemma4.1}. Let $d_{i}$ denote the differentials which change the filtration level by $i$. As before, $d_{0}$ must preserve multi-cycles, so the homology $H(C_{F(n)}(S), d_{0})$ splits over the multi-cycles

\[ H(C_{F(n)}(S), d_{0}) = \bigoplus_{Z} H(C_{F(n)}(Z), d_{0}) \]

We want to compute the complex $C_{F(n)}(Z)$. Recall that $C_{F}(Z)$ was computed to be 
\[ \Big[ \bigotimes_{e_{i} \in Z} R \xrightarrow{\hspace{1mm}U_{i}\hspace{1mm}} R \Big] \otimes \Big[ \bigotimes_{v \in W_{2}(Z)} R \xrightarrow{\hspace{1mm}L(v)\hspace{1mm}} R \Big] \otimes \Big[ \bigotimes_{v \in W_{4}(Z)} R \xrightarrow{\hspace{1mm}Q(v)\hspace{1mm}} R \Big] \otimes \Big[ \bigotimes_{v \in V_{4}(S)} R \xrightarrow{\hspace{1mm}L(v)\hspace{1mm}} R \Big]  \]

\noindent
We can therefore compute $C_{F(n)}(Z)$ by adding the new differentials to this complex. It is not hard to see that the only new discs in $\mathit{CFK}_{n}^{-}(Z)$ correspond to bigons containing $X$ or $XX$ basepoints. For example, let $e_{i}$ be an edge in $Z$, with $x$ and $y$ the two intersection points corresponding to $e_{i}$. When we weren't allowing discs to pass through $X$ or $XX$, the only disc connecting $x$ and $y$ was the bigon containing $U_{i}$. This contributed the tensor summand of 

\[ R \xrightarrow{\hspace{5mm}U_{i}\hspace{5mm}} R \]

\noindent
However, when we allow discs to pass through $X$ and $XX$, we get two new bigons which map from $y$ to $x$, shown in Figure \ref{bivalentbigons}.

\begin{figure}[h!]
\centering
\begin{subfigure}{.5\textwidth}
  \centering
   \begin{overpic}[width=.5\textwidth]{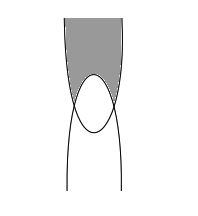} \scriptsize
   \put(40,49){$O_{i}$}
   \put(32.5,49){$\bullet$}
   \put(51,49){$\bullet$}
   \put(29,50){$y$}
   \put(54.8,49.5){$x$}
   \end{overpic}
  \label{fig:sub1}
\end{subfigure}%
\begin{subfigure}{.5\textwidth}
  \centering
  \begin{overpic}[width=.5\textwidth]{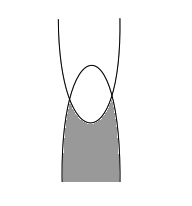} \scriptsize
   \put(40,51){$O_{i}$}
   \put(32,50){$\bullet$}
   \put(52.5,51){$\bullet$}
   \put(28,51){$y$}
   \put(56,51.5){$x$}
  \end{overpic}
  \label{fig:sub2}
\end{subfigure}
\caption{Two new bigons from $y$ to $x$}
\label{bivalentbigons}
\end{figure}
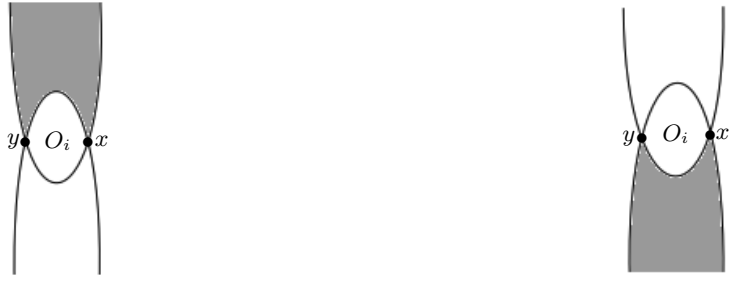

The type of contribution from these bigons depends on whether the endpoint vertices of $e_{i}$ are 2-valent or 4-valent, which is why we only showed the local portion of the bigons in Figure \ref{bivalentbigons}. However, in either case the contribution has a coefficient of degree $n$. We will denote this coefficient by $p(e_{i})$ (the precise polynomial will not be relevant for our computations). The tensor summand then becomes

\[ R \mathrel{ \substack{\xrightarrow{\hspace{5mm}U_{i} \hspace{5mm}} \\[-.4ex] \xleftarrow[\hspace{3.5mm}p(e_{i}) \hspace{3.5mm}]{}}} R \]

For a vertex $v$ in $W_{2}(Z)$, there are two intersection points $x$ and $y$ corresponding to $v$. In $C_{F}(Z)$, they contributed a tensor summand of 

\[ R \xrightarrow{\hspace{5mm}L(v)\hspace{5mm}} R \]

\noindent
In $C_{F(n)}(Z)$, we have an extra differential corresponding to the bigon from $y$ to $x$ through $X$         (See Figure \ref{bivbigon}). Since $X$ carries a coefficient of $u_{1}(v)$, the summand becomes

\[ R \mathrel{ \substack{\xrightarrow{\hspace{5mm}L(v) \hspace{5mm}} \\[-.4ex] \xleftarrow[\hspace{4.5mm}u_{1}(v) \hspace{4.5mm}]{}}} R \]

Similarly, for 4-valent vertices $v$ in $W_{4}(Z)$, $C_{F}(Z)$ contains a tensor summand

\[ R \xrightarrow{\hspace{5mm}Q(v)\hspace{5mm}} R \]

\noindent
In $C_{F(n)}(Z)$, we have en extra differential corresponding to the bigon through $XX$ shown in Figure \ref{4bigon}. The $XX$ contributes a coefficient of $u_{2}(v)$, so the summand becomes

\[ R \mathrel{ \substack{\xrightarrow{\hspace{5mm}Q(v) \hspace{5mm}} \\[-.4ex] \xleftarrow[\hspace{4.5mm}u_{2}(v) \hspace{4.5mm}]{}}} R \]

\begin{figure}[h!]
\centering
\begin{subfigure}{.4\textwidth}
  \centering
  \begin{overpic}[width=.7\textwidth]{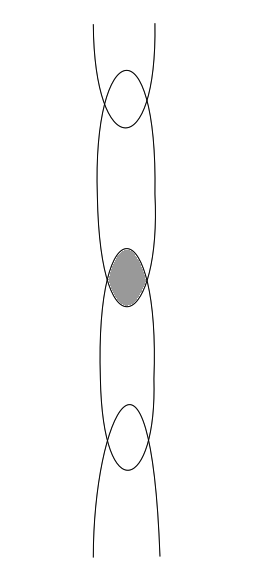} \scriptsize
   \put(19.8,81.5){$O_{1}$}
   \put(20,22.5){$O_{2}$}
   \put(20.3,50.5){$X$}
   \put(13.5,66){$\alpha$}
   \put(13.5,93){$\beta$}
   \put(14,35){$\beta$}
   \put(13.5,5){$\alpha$}
   \put(17.9,50.7){$\bullet$}
   \put(24.7,50.7){$\bullet$}
   \put(16,51){$y$}
   \put(26.5,51){$x$}
   \end{overpic}
   \caption{The bigon through $X$}
    \label{bivbigon}
\end{subfigure}%
\begin{subfigure}{.6\textwidth}
  \centering
   \begin{overpic}[width=.9\textwidth]{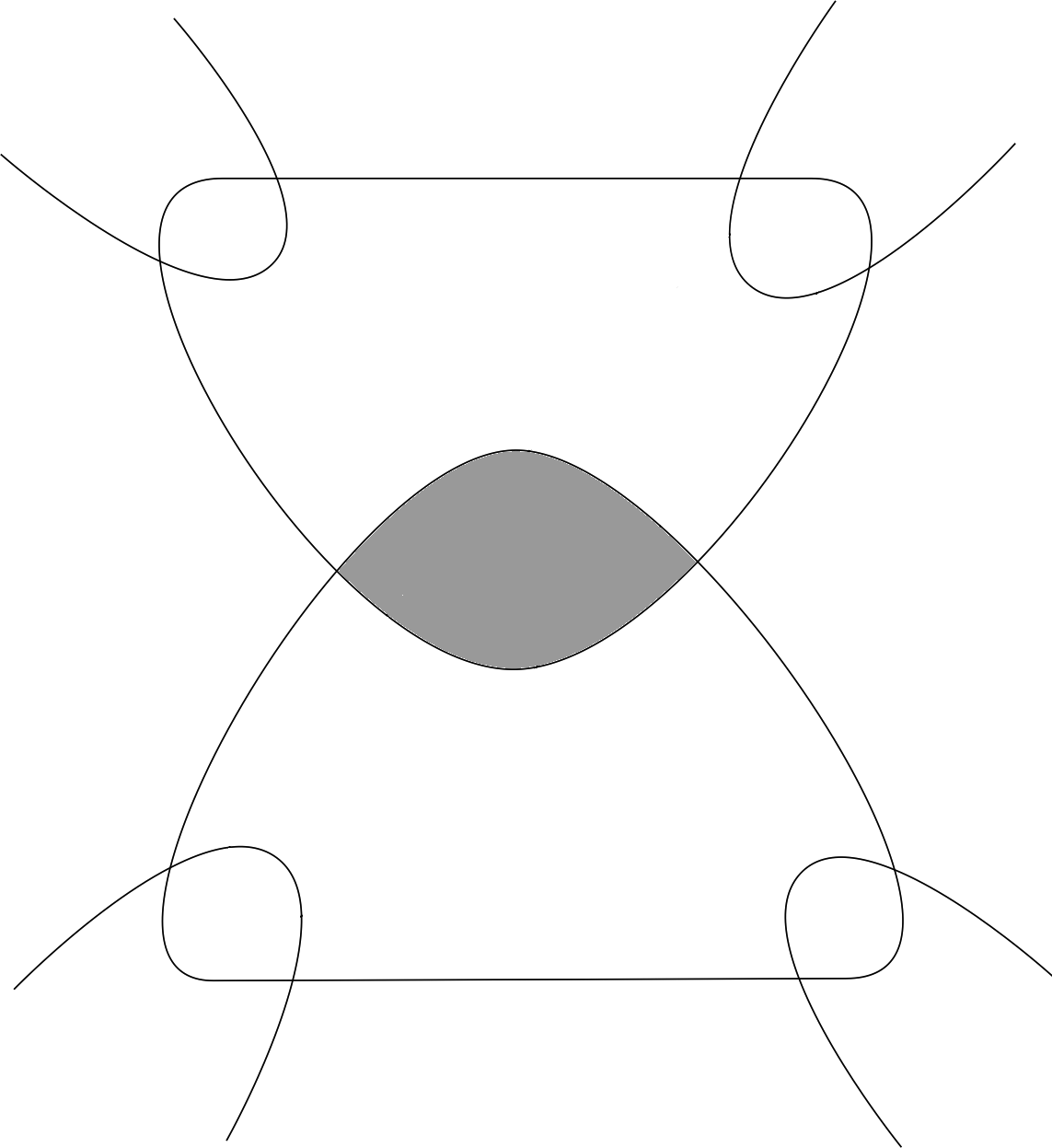} \scriptsize
   \put(17.6,79){$O_{1}$}
   \put(68,79){$O_{2}$}
   \put(19,19.5){$O_{3}$}
   \put(72,19){$O_{4}$}
   \put(41.5,50){$XX$}
   \put(28.5,49.35){$\bullet$}
   \put(60,50.25){$\bullet$}
   \put(26.3,49.8){$y$}
   \put(62,50.5){$x$}
   \put(43,86){$\alpha$}
   \put(1,17){$\alpha$}
   \put(89,18){$\alpha$}
   \put(1,81){$\beta$}
   \put(87,83){$\beta$}
   \put(45,11.5){$\beta$}
   \end{overpic}
   \caption{The bigon through $XX$}
    \label{4bigon}
\end{subfigure}
\caption{New differentials passing through $X$ and $XX$}
\label{fig:test}
\end{figure}
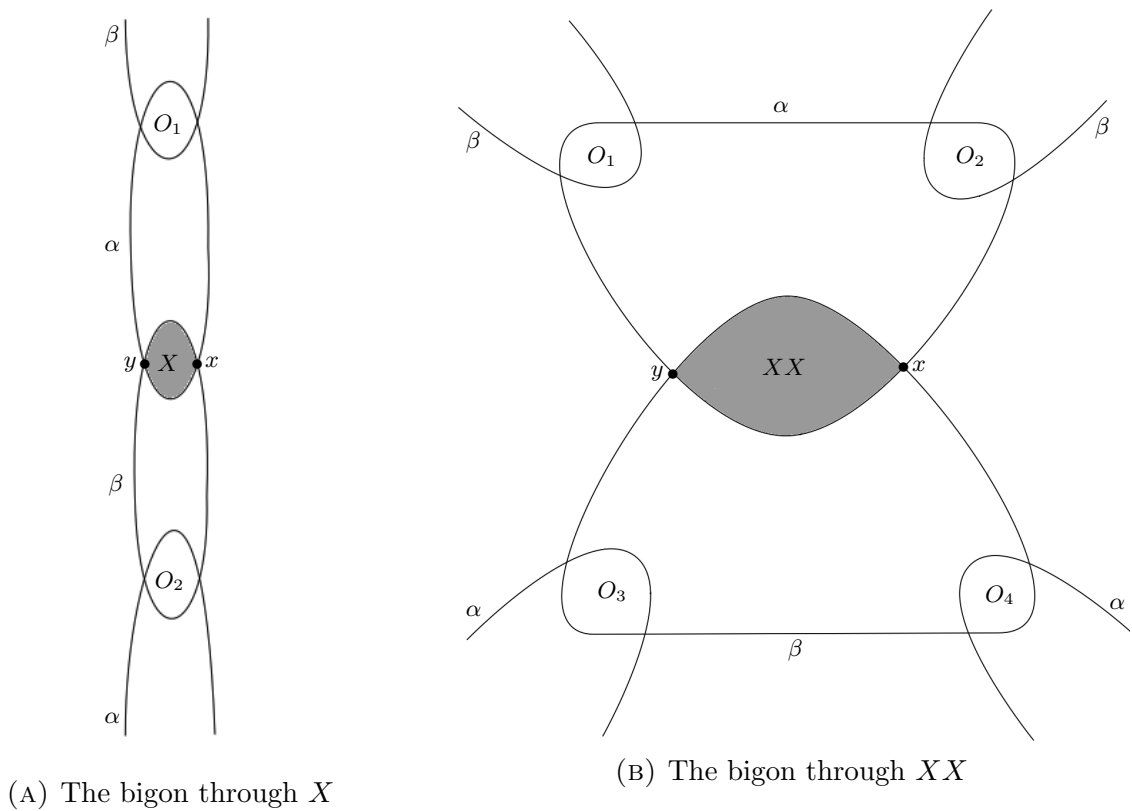

Finally, the complex $K(S)$ gets changed to $K_{n}(S)$, so the whole complex for $C_{F(n)}(Z)$ can be written as 

\[ \Big[ \bigotimes_{e_{i} \in Z} R \mathrel{ \substack{\xrightarrow{\hspace{1.5mm}U_{i} \hspace{1.5mm}} \\[-.4ex] \xleftarrow[\hspace{0mm}p(e_{i}) \hspace{0mm}]{}}} R \Big] \otimes \Big[ \bigotimes_{v \in W_{2}(Z)} R \mathrel{ \substack{\xrightarrow{\hspace{1mm}L(v) \hspace{1mm}} \\[-.4ex] \xleftarrow[\hspace{0.5mm}u_{1}(v) \hspace{0.5mm}]{}}} R \Big] \otimes \Big[ \bigotimes_{v \in W_{4}(Z)} R \mathrel{ \substack{\xrightarrow{\hspace{1mm}Q(v) \hspace{1mm}} \\[-.4ex] \xleftarrow[\hspace{0.5mm}u_{2}(v) \hspace{0.5mm}]{}}} R \Big] \otimes \Big[ \bigotimes_{v \in V_{4}(S)} R \mathrel{ \substack{\xrightarrow{\hspace{1mm}L(v) \hspace{1mm}} \\[-.4ex] \xleftarrow[\hspace{0.5mm}u_{1}(v) \hspace{0.5mm}]{}}} R \Big]  \]

Now that we have our complex computed, we want to compare its homology with $H_{n}(S-Z)$. We will denote the differentials which do not pass through any $X$ or $XX$ basepoints $d_{0+}$, and the new differentials $d_{0-}$. Observe that with respect to the bigrading $(gr_{q},gr_{h})$ introduced in Section \ref{gradings}, $d_{0+}$ has bigrading $\{2,2\}$, and $d_{0-}$ has bigrading $\{2n, -2\}$. 

\begin{lem}

Up to an overall grading shift, the homology $H_{*}(H_{*}(C_{F(n)}(Z),d_{0+}), d_{0-}^{*})$ is isomorphic to $H_{*}(H_{*}(C_{n}(S-Z), d_{+}), d_{-}^{*})$

\end{lem}

\begin{proof}

It follows from Theorem \ref{bigthm} that $H(C_{F(n)}(Z),d_{0+}) \cong H_{H}(S-Z)$. To complete the proof, we need to show that $d_{0-}^{*}$ corresponds to the $d_{-}$ differential under this isomorphism. For vertices which are 2-valent in both $S$ and $S-Z$ (i.e. $v \in W_{2}(Z)$), this is obvious. The same is true for vertices which are 4-valent in both $S$ and $S-Z$ (i.e. $v \in W_{4}(Z)$).

The only identification which is non-trivial is that the $d_{0-}$ differential corresponding to a vertex which is 4-valent in $S$ but 2-valent in $S-Z$ is the same as the $d_{-}$ differential on the 2-valent vertex in $S-Z$. Let $e_{i},e_{j}$ be the outgoing edges at $v$ and $e_{k},e_{l}$ the incoming edges. The multi-cycle $Z$ must contain one incoming and one outgoing edge - without loss of generality, assume $Z$ contains $e_{i}$ and $e_{k}$. The coefficient of the $d_{-}$ differential is given by 

\[ \frac{U_{j}^{n+1}-U_{l}^{n+1}}{U_{j}-U_{l}} \]

\noindent
while the coefficient of the $d_{0-}$ differential is given by 

\[ \frac{U_{i}^{n+1}+U_{j}^{n+1}-U_{k}^{n+1}-U_{l}^{n+1} - Q(v)u_{2}(v)}{U_{i}+U_{j}-U_{k}-U_{l}} \]

\noindent
Recall that to achieve the isomorphism in Theorem \ref{bigthm}, we first cancelled the Koszul complex on the $U_{i}$ for $e_{i}$ in $Z$, as these elements formed a regular sequence. We therefore want to show that these two coefficients are equal in $R_{Z} = R/\{U_{i} = 0 \text{ for } e_{i}  \in Z\} $. Substituting $U_{i}=U_{k}=0$ into the above equation and noting that this causes $Q(v)$ to be zero, we get the desired equality.
\end{proof}

Define $H^{\pm}(C_{F(n)}(S)) = H_{*}(H_{*}(C_{F(n)}(S),d_{0+}), d_{0-}^{*})$. Since both $d_{0+}$ and $d_{0-}$ are homogeneous with respect to the bigrading, this homology is bigraded as well. Applying the lemma and adding in the gradings from Section \ref{gradings}, we see that 
\begin{equation} \label{eqn6}
H^{\pm}(C_{F(n)}(S)) \cong \bigoplus_{Z} H^{\pm}(C_{n}(S-Z)) \{ -2r(S-Z) + T_{2}(Z) - T_{1}(Z),2r(Z)\} 
\end{equation}

Recall from Lemma \ref{gradinglemma} that $H^{\pm}(C_{n}(S-Z))$ lies in a single horizontal grading, namely $2r(S-Z)$. Adding in the shift, the homology corresponding to a multi-cycle $Z$ must lie in horizontal grading $2r(S-Z)+ 2r(Z) = 2r(S)$. But this does not depend on $Z$, so we have shown the following:

\begin{lem}
The homology $H^{\pm}(C_{n}(S))$ lies in a single horizontal grading.
\end{lem}

The original differentials on $C_{F}(S)$ all have bigrading $\{2,2\}$. The new differentials on $\mathit{CFK}_{n}^{-}(S)$ have bigrading $\{2+2k(n-1), 2-4k\}$, where $k$ is the sum of the multiplicities of the holomorphic discs at all $X$ and $XX$ markings. The new differentials on $K_{n}(S)$ all have bigrading $\{2n, -2\}$. Thus, all differentials on $C_{n}(S)$ change the horizontal grading by 2 (mod 4). This implies that no induced differentials can have horizontal grading 0, which tells us that the remaining differentials on our complex are all trivial, giving us the following:

\begin{lem}

The total homology $H_{*}(C_{F(n)}(S),d)$ is isomorphic to $H^{\pm}(C_{F(n)}(S))$.

\end{lem}

This isomorphism is singly-graded with grading $gr_{n} = gr_{q}+(n-1)gr_{h}/2$, as the total differential on $C_{F(n)}(S)$ is homogeneous of degree $n+1$ with respect to this grading.

Going back to (\ref{eqn6}), we know that as bigraded vector spaces, we have the isomorphism 
\[  H^{\pm}(C_{F(n)}(S)) \cong \bigoplus_{Z} H^{\pm}(C_{n}(S-Z)) \{ -2r(S-Z) + T_{2}(Z) - T_{1}(Z),2r(Z)\}  \] 

\noindent
We can use Lemma \ref{sl1} in the same way as in the previous section to exchange the cycle notation for labelings by tensoring with $sl_{1}$ homology.
\[  H^{\pm}(C_{F(n)}(S)) \cong \bigoplus_{f} H_{1}(S_{f,1}) \otimes H_{n}(S_{f,2}) \{ -2r(S_{f,2}) + T_{2}(S_{f,1}) - T_{1}(S_{f,1}), 0 \}  \] 

\noindent
Applying the bigraded composition product formula (\ref{wagnercomp}), this gives an isomorphism of bigraded vector spaces

\begin{equation} \label{slnlemma}
H^{\pm}(C_{F(n)}(S)) \cong H_{n+1}(S)  
\end{equation}

Since $H_{*}(C_{F(n)}(S),d) \cong H^{\pm}(C_{F(n)}(S))$ as graded vector spaces with grading $gr_{n}$, this gives an isomorphism

\[ H_{*}(C_{F(n)}(S),d) \cong H_{n+1}(S) \]

\noindent
where we are viewing $H_{n+1}(S)$ as singly graded, with grading $gr_{n}$. Since singly graded $sl_{n+1}$ homology is typically viewed with respect to the grading $gr_{n+1}$, this isn't quite what we want. Fortunately, since the homology is concentrated in horizontal grading $gr_{h} = 2r(S)$, we see that $gr_{n+1} = gr_{n} + r(S)$. Thus, up to an overall grading shift, the formula remains true in the standard grading $gr_{n+1}$.

\begin{thm}

Up to an overall grading shift, the total homology $H_{*}(C_{F(n)}(S),d)$ is isomorphic to $H_{n+1}(S)$. 

\end{thm}

\begin{cor}

For all $n \ge 1$, there is a spectral sequence whose $E_{1}$ page is $H_{F}(S)$ which converges to $H_{n+1}(S)$.

\end{cor}

\begin{proof}

All of the original differentials on $C_{F}(S)$ have Alexander grading $0$. The new differentials on $\mathit{CFK}^{-}_{n}$ have Alexander grading $k(-n-1)$, where $k$ is the sum of the multiplicities of the disc at the $X$ and $XX$ basepoints, and the new differentials on the Koszul complex have Alexander grading $-n-1$. In particular, all of the new differentials strictly decrease the Alexander grading, so it induces a filtration with respect to which the filtered homology is $H_{*}(C_{F}(S))$. Thus, the corresponding spectral sequence has $E_{1}$ page $H_{F}(S)$, and converges to the total homology $H(C_{F(n)}(S)) \cong H_{n+1}(S)$.
\end{proof}

\subsection{Proof of Manolescu's Conjecture}

At this point, we have three spectral sequences: one from HOMFLY-PT homology to $sl_{n}$ homology, one from HOMFLY-PT homology to knot Floer homology, and one from knot Floer homology to $sl_{n}$ homology. Diagrammatically, this looks like

\begin{figure}[H]
\centering
\begin{tikzpicture}
\node(HOMFLY){$H_{H}(S)<1>$};
\node (HFK) at ([shift={(5,0)}]HOMFLY) {$H_{F}(S)$};
\node (sln) at ([shift={(2.7,-2.3)}]HOMFLY) {$H_{n}(S)$};
\path[-stealth] 
  (HOMFLY) edge (HFK)
  (HOMFLY) edge (sln)
  (HFK) edge (sln);
\end{tikzpicture}
\end{figure}

\noindent
where the arrows correspond to spectral sequences. Since both HOMFLY-PT homology and knot Floer homology have spectral sequences going to $sl_{n}$ homology for all $n \ge 2$, it is clear that they have a great deal in common. The conjecture of Manolescu is that they are in fact isomorphic. Since we have a spectral sequence from HOMFLY-PT homology to knot Floer homology, this is equivalent to the spectral sequence collapsing at the $E_{1}$ page. In this section, we will prove Manolescu's conjecture.

\begin{thm}
The HOMFLY-PT homology $H_{H}(S)<1>$ and knot Floer homology $H_{F}(S)$ are isomorphic as bigraded vector spaces.
\end{thm}

\begin{proof}

We will start with the complex $C_{F(n)}$ from the previous section. There are two filtrations defined on this complex so far - the one induced by the Alexander grading, and the one induced by the basepoints. Consider the associated bigraded object from these two filtrations.

The differentials always change the Alexander grading by a multiple of $n+1$, so let $d_{ij}$ denote those differentials which change the Alexander grading by $i(n+1)$ and change the basepoint grading by $j$. 

Theorem \ref{filterediso} states that 

\[ H_{*}(C_{F(n)}, d_{00}) \cong H_{H}(S)<1> \]

\noindent
Since $H_{F}(S) \cong H_{*}(C_{F(n)}, d_{0*})$, our theorem is equivalent to $d_{0k}^{*}$ being zero on \linebreak $H_{*}(C_{F(n)}, d_{00})$. We will prove this by contradiction. In particular, suppose that some $d_{0k}^{*}$ is non-zero, and let $a$ denote the smallest such $k$. 

Because $a$ is minimal, it is clear that $d_{0a}^{*}$ and $d_{10}^{*}$ anti-commute, as $d^{*}_{0a} \circ d^{*}_{10}$ and $d^{*}_{10} \circ d^{*}_{0a}$ are the only components of $d^{2}$ which change the basepoint filtration by $a$ and the Alexander filtration by $n+1$. These are the differentials that we will be interested in, so we will rename them. We will write $d_{F}$ instead of $d^{*}_{0a}$ and $d_{n}$ instead of $d^{*}_{10}$, since $d^{*}_{10}$ depends on $n$. The differentials $d_{F}$ and $d_{n}$ both act on $H_{*}(C_{F(n)}, d_{00})$, so using the above isomorphism we will view them as acting on $H_{H}(S)<1>$.

We know from (\ref{slnlemma}) that there is an isomorphism of bigraded vector spaces

\[ H_{*}(H_{H}(S)<1>, d_{n}) \cong H_{n+1}(S) \]

To summarize our setup, we have a family of differentials $d_{n}$ on $H_{H}(S)<1>$, each having bigrading $(2n, -2)$, such that the homology with respect to each is $H_{n+1}$, and there is a differential $d_{F}$ on $H_{H}(S)<1>$ with bigrading $(2,2)$ which is non-trivial and anti-commutes with each $d_{n}$.

We know that the smallest horizontal grading in which $H_{H}(S)<1>$ is non-trivial is $2r(S)$, and that the homology $H_{*}(H_{H}(S)<1>, d_{n})$ lies only in this horizontal grading. Let $h_{min}$ be the minimal horizontal grading on which $d_{F}$ is non-zero, and let $x$ be an element of $H_{H}(S)<1>$ in bigrading $(q,h_{min})$ for some $q$ with $d_{F}(x) \ne 0$. Define $d_{F}(x)=y$. 

We know also that $H_{H}(S)<1>$ is bounded below in quantum grading and $d_{n}$ changes the quantum grading by $2n$, so choose $N$ sufficiently large that $y$ can not be in the image of $d_{N}$. Since $y$ lies in horizontal grading $h_{min}+2$ (in particular, not $h_{min}$), it follows that $y \in Ker(d_{N})$ iff $y \in Im(d_{N})$. Thus, $y$ is not in the kernel of $d_{N}$.

But then $ d_{N} \circ d_{F}(x)$ is non-zero, while $d_{F} \circ d_{N}(x)$ must be zero because $d_{N}(x)$ lies in horizontal grading $h_{min}-2$, and $d_{F}=0$ for all horizontal gradings less than $h_{min}$, which contradicts the fact that $d_{N}$ and $d_{F}$ anticommute.
\end{proof}

We relate this theorem to Manolescu's conjecture with the following corollary.

\begin{cor}

There is an isomorphism of Tor groups

\[ \mathit{Tor}_{R}(R/L, R/N) \cong \mathit{Tor}_{R}(R/L, R/Q)  \]

\noindent
as bigraded vector spaces. The bigrading on $\mathit{Tor}_{R}(R/L, R/N)$ is given by $(q, h)$ where $q$ is the quantum grading coming from the polynomial ring and $h$ is the homological grading, and the bigrading on $\mathit{Tor}_{R}(R/L, R/Q)$ is $(q+h, h)$.
\end{cor}

\begin{proof}

The difference between our complex $C_{F}(S)$ and Manolescu's knot Floer complex for $S$ is that we have added an additional unknotted component at infinity, placed the marked edge on that component, and reduced that component (i.e. set that $U_{i}$ equal to zero).

Manolescu instead placed the marked edge on the leftmost strand of the braid. Let $C^{M}_{F}(S)$ denote Manolescu's complex. We know that adding an unknotted component and reducing it doubles the homology. In particular, with respect to the $(q,h)$ bigrading there is an isomorphism

\[ H_{F}(S) \cong H^{M}_{F}(S) \otimes V  \]

\noindent
where $V = \Q \{-1, -1\} \oplus \Q \{1, 1\}$.

Similarly, the middle HOMFLY-PT homology can be viewed as the unreduced HOMFLY-PT homology of the 1-1 tangle obtained  by breaking an edge in the diagram. Let $C_{H}^{M}(S)$ denote the middle HOMFLY-PT homology of $S$. Then, with respect to the $(q,h)$ grading, we have the following isomorphism

\[ H_{H}(S) \cong H^{M}_{H}(S) \{0,-1\} \oplus H^{M}_{H}(S) \{0,1\} \]

\noindent
But when we switch to the $(q+h, h)$ grading, this becomes 

\[ H_{H}(S)<1> \cong H^{M}_{H}(S)<1> \otimes V  \]

The previous theorem states that $H_{F}(S) \cong H_{H}(S)<1>$. Using the above arguments, this becomes

\[  H^{M}_{F}(S) \otimes V \cong H^{M}_{H}(S)<1> \otimes V  \] 

\noindent
as bigraded groups.

Since all of our theories are bounded in $h$-grading and bounded below in $q$-grading, the above isomorphism implies an isomorphism without tensoring with $V$.

\[  H^{M}_{F}(S) \cong H^{M}_{H}(S)<1>   \]

But Manolescu showed that $H^{M}_{F}(S) \cong \mathit{Tor}_{R}(R/L, R/N)$ and 

\noindent
$H^{M}_{H}(S) \cong \mathit{Tor}_{R}(R/L, R/Q)$, which proves the corollary.
\end{proof}

Another significant corollary of this result relates to the $E_{2}$ page of the spectral sequence on $C_{F}(D)$ induced by the cube filtration. This is the page which was conjectured by Manolescu to give HOMFLY-PT homology. In \cite{Me}, we showed that the graded Euler characteristic of the homology

\[ E_{2}^{f}(D) = H_{*}(H_{*}(C_{F}(D), d_{0}^{f}), (d_{1}^{f})^{*})  \]

\noindent
is the HOMFLY-PT polynomial, where $d_{i}^{f}$ denotes the component of the differential on $C_{F}(D)$ which increases the cube grading by $i$ and preserves the basepoint filtration. In particular, we define the triple grading on this complex by the $i$, $j$, and $j$ gradings, where $k$ denotes twice the cube grading, and 

\[ i = 2A-2M-k \]
\[ j = 4A-2M-k \]

\noindent
where $M$ and $A$ are the Maslov and Alexander gradings, respectively. With respect to this triple grading, we showed that 

\[ \sum_{i,j,k} (-1)^{(k-j)/2} dim(E_{2}^{f}(D)^{i,j,k} ) = P_{H}(aq, q, D) \]

 But we have just seen that $H_{*}(C_{F}(D), d_{0}^{f}) \cong H_{*}(C_{F}(D), d_{0})$, so there is a spectral sequence from $H_{*}(H_{*}(C_{F}(D), d_{0}^{f}), (d_{1}^{f})^{*})$ to the $E_{2}$ page $H_{*}(H_{*}(C_{F}(D), d_{0}), d_{1}^{*})$. But all of these differentials have triple grading $\{0, 0, 2\}$, so they do not change the Euler characteristic. Thus we have shown the following:
 
\begin{cor}

Let $C_{F}(D)$ denote the oriented cube of resolutions complex for a braid diagram $D$, and let $E_{2}(D)$ denote the $E_{2}$ page of the spectral sequence on $C_{F}(D)$ induced by the cube filtration. Then the graded Euler characteristic of $E_{2}(D)$ with the triple grading given above is the HOMFLY-PT polynomial $P_{H}(aq, q, D)$.

\end{cor}

\bibliography{TriplyGradedHomology}{}
\bibliographystyle{plain}

\end{document}